\documentclass[10pt,letterpaper]{article}

\usepackage{amsmath,amssymb,amsthm}
\usepackage{mathtools}
\usepackage{stmaryrd} 
\usepackage{enumitem}
\usepackage[letterpaper,left=4cm,right=3cm,top=3cm,bottom=3cm]{geometry}
\usepackage{hyperref}

\usepackage{tikz}
\usetikzlibrary{arrows.meta,decorations.pathreplacing,calc}
\usepackage{pgfplots}
\pgfplotsset{compat=1.18}

\usepackage{todonotes}
\newcommand{\NN}{\ensuremath{\mathbb N}}
\newcommand{\ZZ}{\ensuremath{\mathbb Z}}
\newcommand{\RR}{\ensuremath{\mathbb R}}

\newcommand{\cA}{\mathcal{A}}
\newcommand{\cB}{\mathcal{B}}
\newcommand{\cG}{\mathcal{G}}
\newcommand{\cJ}{\mathcal{J}}
\newcommand{\cS}{\mathcal{S}}

\DeclarePairedDelimiter{\abs}{\lvert}{\rvert}
\DeclarePairedDelimiter{\floor}{\lfloor}{\rfloor}
\DeclarePairedDelimiter{\ceiling}{\lceil}{\rceil}
\DeclarePairedDelimiter{\multi}{\{\!\{}{\}\!\}}
\DeclarePairedDelimiter{\norm}{\lVert}{\rVert}
\DeclarePairedDelimiter{\tf}{\llbracket}{\rrbracket}
\newcommand{\multibinom}[2]{\left(\kern-0.3em\left(\genfrac{}{}{0pt}{}{#1}{#2}\right)\kern-0.3em\right)}
\DeclareMathOperator{\supp}{supp}
\DeclareMathOperator{\mult}{mult}
\DeclareMathOperator{\vol}{vol}

\newcommand{\dd}{\mathop{}\!\mathrm{d}}

\usepackage[most]{tcolorbox}

\colorlet{thmback}{green!10!white}  
\colorlet{thmedge}{green!65!black}
\colorlet{lemback}{green!5!white}   
\colorlet{lemedge}{green!45!black}
\colorlet{defback}{blue!5!white}    
\colorlet{defedge}{blue!45!black}
\colorlet{qstback}{red!5!white}     
\colorlet{qstedge}{red!45!black}
\colorlet{ideaback}{yellow!12!white}  
\colorlet{ideaedge}{orange!60!black}
\colorlet{remarkback}{white}  
\colorlet{remarkedge}{black}

\tcbset{
  tagbox/.style 2 args={
    enhanced, 
    colback=#1, colframe=#2,
    boxrule=0.5pt, arc=6pt,
    left=6pt, right=6pt, top=9pt, bottom=6pt,
    before skip=14pt plus 2pt, after skip=12pt plus 2pt,
    attach boxed title to top left={xshift=10pt, yshift=-\tcboxedtitleheight/2},
    boxed title style={
      colback=#1, colframe=#2, arc=5pt, boxrule=1.0pt,
      left=6pt, right=6pt, top=1pt, bottom=1pt,
    },
    fonttitle=\bfseries\normalsize, coltitle=black,
    description delimiters parenthesis,
    description font=\mdseries,
    separator sign={},
    terminator sign={},
    fontupper=\itshape,
  },
  tagbox upright/.style 2 args={tagbox={#1}{#2}, fontupper=\upshape},
}
\newtcbtheorem{thm}{Theorem}{tagbox={thmback}{thmedge}}{thm}
\newtcbtheorem[use counter from=thm]{lem}{Lemma}{tagbox={lemback}{lemedge}}{lem}
\newtcbtheorem[use counter from=thm]{cor}{Corollary}{tagbox={lemback}{lemedge}}{cor}
\newtcbtheorem[use counter from=thm]{dfn}{Definition}{tagbox upright={defback}{defedge}}{dfn}
\newtcbtheorem[use counter from=thm]{qst}{Question}{tagbox upright={qstback}{qstedge}}{qst}
\newtcbtheorem[use counter from=thm]{cnj}{Conjecture}{tagbox upright={qstback}{qstedge}}{conj}
\newtcbtheorem[use counter from=thm]{idea}{Idea}{tagbox upright={ideaback}{ideaedge}}{idea}
\newtcbtheorem[use counter from=thm]{rem}{Remark}{
    tagbox upright={remarkback}{remarkedge},
    breakable,
    parbox=false
}{rem}
\newtheorem{claim}{Claim}

\begin{document}
  \title{On the thickness of infinite generalized {S}idon sets, {II}: \\ \texorpdfstring{$B_h$}{B sub h}-sets in \texorpdfstring{$\ZZ^d$}{Z to the d}}
  \author{Kevin O'Bryant\thanks{Email: \texttt{kevin.obryant@csi.cuny.edu}.\\ 2020 Mathematics Subject Classification: 11B83, 05B10, 11B05.}}
  \date{November 7, 2026}
  \maketitle

\begin{abstract}
Let $d,h$ be positive integers with $h$ even. A set $\cA\subseteq\ZZ^d$ is a $B_h$-set if the sums of $h$ elements of $\cA$ are distinct up to the order of the summands. 
Let $K\subseteq\RR^d$ be a convex body containing the origin, and let $A_K(n)$ be the number of elements of $\cA$ in $\bigcup_{t<n}t\ast K$. We prove that if $h$ is even then
\[
  \liminf_{n\to\infty} \frac{A_K(n)}{\left( \vol(n\ast K)/\log n \right)^{1/h}}
  \le \left(\frac{4}{d\log 2}
        \cdot \frac{\Gamma(1+h/2)^2\Gamma(1+d/2)^2}{\Gamma(1+d/h)^{h}}\right)^{1/h}.
\]
The constant does not depend on $K$. As a special case, this explicitly bounds\linebreak $\liminf_n A(n) \left(\frac{\log n}{n^d}\right)^{1/h}$, where $A(n)=\abs{\cA\cap[0,n)^d}$, which Rackham~\&~{\v{S}}arka (2010) had proved finite.
\end{abstract}

\section{Introduction}\label{sec:intro}
A set $\cA \subseteq \ZZ^d$ is a $B_h$-set if the $h$-fold sums
  \[\bar a_1+\cdots+\bar a_h, \qquad \bar a_i\in\cA,\]
are distinct up to rearrangement; that is, if $\bar a_1+\cdots+\bar a_h=\bar b_1+\cdots+\bar b_h$ with all $\bar a_i,\bar b_i\in\cA$, then the multisets of summands are identical: $\multi{\bar a_1,\dots,\bar a_h}=\multi{\bar a_1',\dots,\bar a_h'}$.
A $B_2$-set is called a Sidon set in number theory, a Golomb ruler in graph theory and computational number theory, and a Babcock set in IEEE journals. For a positive integer $k$, the set of all $k$-fold sums of elements of $\cA$ is written $k\cA$.

Throughout this work, $d\ge 1$ is a fixed integer, $\NN=\{0,1,2,\dots\}$, and a bar indicates an element of $\RR^d$. Unbarred letters are scalars, sets, or multisets. For $\bar x=(x_1,\dots,x_d)\in\RR^d$ we write $\norm{\bar x}\coloneqq\max_i\abs{x_i}$.

A \emph{convex body} is a compact convex subset of $\RR^d$ with nonempty interior. Throughout, $K$ is a convex body with $\bar 0\in K$, with volume $V \coloneqq \vol(K)$, and gauge $\rho=\rho_K$ (defined momentarily, and sometimes called the Minkowski functional induced by $K$). We specifically do not require $\bar 0$ to be an interior point of $K$, so that the cube $[0,1]^d$ is allowed. For $t\ge0$ we write the dilation of $K$ by $t$ as $t\ast K\coloneqq\{t\bar x : \bar x\in K\}$, so that $\vol(t\ast K)=t^dV$, and since $\bar 0\in K$ and $K$ is convex, the sets $t\ast K$ increase with $t$. The \emph{gauge} of $K$ is the function $\rho=\rho_K$ defined on $\RR^d$ by
  \[\rho(\bar x)\coloneqq\inf\{t>0 : \bar x\in t\ast K\}\in[0,\infty],\]
and $\inf\emptyset=\infty$. For $\cA\subseteq\ZZ^d$ we set
  \[A_K(n)\coloneqq\#\{\bar a\in\cA : \rho(\bar a)<n\}=\#\Bigl\{\bar a\in\cA : \bar a\in\textstyle\bigcup_{t<n}t\ast K\Bigr\}.\]
For $K=[0,1]^d$ the gauge is 
  \[\rho(\bar x)=
  \begin{cases} \norm{\bar x}, & \bar x \in [0,\infty)^d, \\
    \infty, & \bar x \in \RR^d \setminus [0,\infty)^d,
  \end{cases}\]
and for $\cA\subseteq\ZZ^d$ the counting function $A_K(n)$ is the usual counting function (in which case we drop the $K$ from the notation)
  \[A_K(n) = A(n)= \abs*{\cA\cap[0,n)^d}.\]
For $K$ the Euclidean unit ball, $A_K(n)$ counts the elements of $\cA$ of Euclidean length strictly less than $n$.

\begin{figure}[t]
\centering
\begin{tikzpicture}
  \begin{axis}[
      width=\textwidth,
      height=0.3333\textwidth,
      xtick={2,4,6,8,10,12,14},
      xmin=1, xmax=15,
      ymin=0, ymax=6,
      ytick={0,1,2,3,4,5,6},
      axis lines=left,
      xlabel={$h$},
      xlabel style={at={(axis description cs:1,0)}, anchor=west, xshift=2pt},
      yticklabel style={xshift=-2pt},
      legend columns=-1,
      legend style={at={(0.03,0.97)}, anchor=north west, draw=none, fill=none, font=\footnotesize, /tikz/every even column/.append style={column sep=1em}},
      legend cell align=left,
      tick label style={font=\footnotesize},
      label style={font=\footnotesize},
    ]
    \addplot[color=blue!65!black, mark=*, mark size=1.5pt, thick, smooth] coordinates {
        (2, 2.4022) (4, 2.2765) (6, 2.5197) (8, 2.8390) (10, 3.1851) (12, 3.5434) (14, 3.9079)
      };
    \addlegendentry{$d=1$}

    \addplot[color=red!65!black, mark=square*, mark size=1.5pt, thick, smooth] coordinates {
        (2, 1.6986) (4, 2.0798) (6, 2.4280) (8, 2.7878) (10, 3.1545) (12, 3.5250) (14, 3.8974)
      };
    \addlegendentry{$d=2$}

    \addplot[color=green!45!black, mark=triangle*, mark size=1.8pt, thick, smooth] coordinates {
        (2, 1.3869) (4, 2.0894) (6, 2.5142) (8, 2.9015) (10, 3.2806) (12, 3.6575) (14, 4.0336)
      };
    \addlegendentry{$d=3$}

    \addplot[color=gray!55!black, dashed, thick, domain=9:15, samples=25]
      {(x + ln(x))/(2*2.718281828459045)};
    \addlegendentry{$(h+\log h)/(2e)$}
  \end{axis}
\end{tikzpicture}
\caption{The constant $C_{d,h}$ of Definition~\ref{dfn:C} and Theorem~\ref{thm:main} for $d=1,2,3$, together with its asymptotic main terms $(h+\log h)/(2e)$, which do not depend on $d$. For $d=1$ the minimum is, strangely, at $h=4$. For $d\ge2$ it is at $h=2$, where the constant is $C_{d,2}=2/\sqrt{d\log2}$.}
\label{fig:constant-vs-h}
\end{figure}

\begin{dfn}{$C_{d,h}$}{C}
For positive integers $d,h$, set
  \(\displaystyle 
  C_{d,h}\coloneqq\left(\frac{4}{d\log 2} \cdot \frac{\Gamma(1+\tfrac h2)^2\Gamma(1+\tfrac d2)^2}{\Gamma(1+\tfrac dh)^{h}}\right)^{1/h}.\)
\end{dfn}

\begin{thm}{Main Theorem}{main}
Let $d,h$ be a positive integers with $h$ even, let $K\subseteq\RR^d$ be a convex body with $\bar 0\in K$, and let $\cA\subseteq\ZZ^d$ be a $B_h$-set.
Then
  \[\liminf_{n\to\infty} \frac{A_K(n)}{\left(\vol(n\ast K)/\log n\right)^{1/h}} \le C_{d,h}.\]
Equivalently, for every $\epsilon>0$ there are infinitely many $n$ with
  \[A_K(n)^h \le (C_{d,h}^h +\epsilon)\frac{\vol(n\ast K)}{\log n}.\]
\end{thm}

Notably, the constant 
  \[
  C_{d,h}^h
  = \underbrace{\frac{4}{d\log 2}\vphantom{\frac{\Gamma(1+h/2)^2}{\Gamma(1+d/h)^h}}}_{\text{energy}}
    \cdot
    \underbrace{\frac{\Gamma(1+h/2)^2}{\Gamma(1+d/h)^h}}_{\text{simplex}}
    \cdot
    \underbrace{\Gamma(1+d/2)^2\vphantom{\frac{\Gamma(1+h/2)^2}{\Gamma(1+d/h)^h}}}_{\text{dimension}}
\] 
does not depend on $K$. The three factors come from different aspects of our proof, and we believe the `simplex' and `dimension' factors are best-possible, while the `energy' factor is not, and might even be $0$. Since $\Gamma(\tfrac32)^2=\pi/4$, we have
  \[C_{1,h}=\left(\frac{\pi}{\log 2}\cdot\frac{\Gamma(1+h/2)^2}{\Gamma(1+1/h)^h}\right)^{1/h},\]
and for Sidon sets $C_{d,2}$ simplifies to $\sqrt{4/(d\log 2)}$ in every dimension. For each fixed $d$, $C_{d,h}=\left(h+\log h\right)/(2e)+O(1)$ as $h\to\infty$. Figure~\ref{fig:constant-vs-h} plots the first values of $C_{1,h}$, $C_{2,h}$, and $C_{3,h}$.

The most natural example of $K$ is the cube, $K=[0,1]^d$. Rackham~\&~{\v{S}}arka~\cite{2010.Rackham&Sarka} proved the following $\liminf$ to be finite, and we provide an explicit constant. With $K=[0,1]^d$, Theorem~\ref{thm:main} gives this corollary.
\begin{cor}{Explicit Rackham \& \v{S}arka}{cube}
Let $d,h$ be positive integers with $h$ even, and let $\cA\subseteq\NN^d$ be a $B_h$-set. Then
  \[\liminf_{n\to\infty} \frac{A(n)}{\left(n^d/\log n\right)^{1/h}} \le C_{d,h}.\]
\end{cor}

\subsection{More corollaries}
Theorem~\ref{thm:main} says that a $B_h$-set cannot be too thick, with a logarithmic-power saving over the ``obvious'' bound (given in Lemma~\ref{lem:easy Bh bound} below). The same implication can be rearranged to say that the elements of a $B_h$-set cannot grow too slowly, and that the empty regions in a $B_h$-set, or in its sumset, cannot all be small. In this subsection we give explicit restatements of Theorem~\ref{thm:main} in those forms.

\begin{cor}{Growth of Elements}{limsup}
  Let $d$, $h$, $K$ be as in Theorem~\ref{thm:main}, and let $\cA=\{\bar a_1,\bar a_2, \dots\} \subseteq\ZZ^d$ be an infinite $B_h$-set with $\rho(\bar a_1)\le\rho(\bar a_2)\le\cdots$. For each $\epsilon>0$ there are infinitely many $n$ with 
  \[\rho(\bar a_n)^d \ge \frac{1-\epsilon}{\vol(K)}\cdot \frac{h}{d C_{d,h}^h} \, \cdot \, n^h \log n.\]
\end{cor}
Such an enumeration exists because each set $\{\bar a\in\cA:\rho(\bar a)<x\}$ is finite (Lemma~\ref{lem:lattice}), and elements $\bar a$ with $\rho(\bar a)=\infty$, if any, come last.
\begin{proof}
Write $V\coloneqq\vol(K)$, as throughout this entire work. If $A_K(x)$ is bounded, then $\rho(\bar a_n)=\infty$ for all large $n$ and there is nothing to prove, so assume $A_K(x)\to\infty$.

Fix $\epsilon>0$ and choose $\delta>0$ small enough so that $\bigl(C_{d,h}/(C_{d,h}+\delta)\bigr)^h\ge1-\epsilon/2$. By Theorem~\ref{thm:main} there are infinitely many integers $x$ with
  \[A_K(x)\le (C_{d,h}+\delta)\left(\frac{Vx^d}{\log x}\right)^{1/h}.\]
For such an $x$, put $n\coloneqq A_K(x)+1$. Since $A_K$ is nondecreasing and unbounded, infinitely many distinct $n$ arise this way. Exactly $n-1$ elements of $\cA$ have $\rho<x$, so $\rho(\bar a_n)\ge x$, and rearranging the display, and taking $n$ sufficiently large, we find 
  \begin{equation}\label{dfn:growth step}
    V\rho(\bar a_n)^d
    \ge Vx^d
    \ge \frac{A_K(x)^h \log x }{(C_{d,h}+\delta)^h}
    = \frac{(n-1)^h}{(C_{d,h}+\delta)^h}\log x
    \ge \left(1-\frac{\epsilon}{2}\right)\,\frac{n^h}{C_{d,h}^h}\log x
  \end{equation}
Taking logarithms of $Vx^d\ge (n-1)^h/(C_{d,h}+\delta)^h$ gives $\log x\ge \left(1-\frac \epsilon2\right) \frac hd\log n$ for sufficiently large $n$. Substituting this into~\eqref{dfn:growth step}, we conclude
  \[V\rho(\bar a_n)^d
  \ge \left(1-\frac{\epsilon}{2}\right)\,\left(1-\frac{\epsilon}{2}\right) \frac{n^h}{C_{d,h}^h} \, \frac hd \log n
  \ge (1-\epsilon)\frac{h}{dC_{d,h}^h}n^h\log n.\qedhere\]
\end{proof}

\begin{dfn}{Cubic Gaps $\gamma_\cB(n)$}{gamma}
For $\cB\subseteq\NN^d$ and $n\ge1$, let $\gamma_\cB(n)$ be the largest volume $g^d$ of a cube $\bar x+[0,g)^d\subseteq[0,n)^d$ that is disjoint from $\cB$.
\end{dfn}
For $d=1$ and $\cB=\{b_1<b_2<\cdots\}$, $\gamma_\cB(n)$ is the largest gap $b_{j+1}-b_j$ with $b_{j+1}<n$ (or $b_1$, if that is larger). The next two corollaries are stated in $\NN^d$ because their proofs use that a sum of elements of $\NN^d$ lying in $[0,n)^d$ has all of its summands in $[0,n)^d$. See Section~\ref{sec:differences} for more detail on this point.

\begin{cor}{Large Gaps in the Sumset of a $B_h$-set}{gaps}
Let $d\ge 1$ be an integer, let $h$ be a positive even integer, and let $\cA\subseteq\NN^d$ be an infinite $B_h$-set. Then the $h$-fold sumset $h\cA$ satisfies
  \[\limsup_{n\to\infty}\frac{\gamma_{h\cA}(n)}{\log n}
  \ge \frac{h!}{C_{d,h}^h}
  .\]
In particular, if $\cA\subseteq\NN$ is a Sidon set and $\cA+\cA=\{s_1<s_2<\cdots\}$, then
  \[\limsup_{j\to\infty}\frac{s_{j+1}-s_j}{\log s_j}\ge\frac{\log 2}{2}.\]
\end{cor}

Erd\H os \& S\'ark\"ozy \& S\'os~\cite[Theorem 5]{1994.Erdos-Sarkozy-Sos} observed that there must be such a constant for the $h=2,d=1$ instance.

\begin{proof}
Suppose, by way of contradiction, that there are $\tau<h!/C_{d,h}^h$ and $n_1$ with $\gamma_{h\cA}(n)\le\tau\log n$ for all $n\ge n_1$. Fix such an $n$ and set $g\coloneqq\floor{(\tau\log n)^{1/d}}+1$, so that $g^d>\tau\log n\ge\gamma_{h\cA}(n)$ and consequently every cube of side $g$ contained in $[0,n)^d$ meets $h\cA$. The cube $[0,n)^d$ contains $\floor{n/g}^d$ pairwise disjoint cubes of side $g$, whence
  \[\abs*{h\cA\cap[0,n)^d}\ge\floor*{\frac ng}^d \ge \left(\frac ng-1\right)^d = (1-o(1))\frac{n^d}{\tau\log n},\]
the last equality because $g=(1+o(1))(\tau\log n)^{1/d}$ and $g=o(n)$.

On the other hand, every element of $h\cA\cap[0,n)^d$ is a sum of $h$ elements of $\cA\cap[0,n)^d$, so by the stars-and-bars count of Section~\ref{sec:multisets}
  \[\abs*{h\cA\cap[0,n)^d}
  \le\abs*{\multibinom{\cA\cap[0,n)^d}{h}}
  = \multibinom{A(n)}{h}
  \le\frac{(A(n)+h)^h}{h!}.\]
Combining the two displays, as $n\to\infty$ we have
  \[A(n)\ge (1-o(1))\left(\frac{h!}{\tau}\right)^{1/h} \left( \frac{n^d}{\log n} \right)^{1/h},\]
and therefore
  \[\liminf_{n\to\infty}\frac{A(n)}{\left(n^d/\log n\right)^{1/h}}\ge\left(\frac{h!}{\tau}\right)^{1/h}>C_{d,h},\]
contradicting Corollary~\ref{cor:cube}. This proves the first claim, and the formula for $h!/C_{d,h}^h$ is immediate from Definition~\ref{dfn:C}.

For the second claim, take $h=2$ and $d=1$, where $h!/C_{1,2}^2=2\cdot\frac{\log 2}{4}=\frac{\log 2}{2}$ since $\Gamma(1+\tfrac12)^2$ appears in both numerator and denominator. If $s_{j+1}-s_j\le\tau\log s_j$ for all $j\ge j_0$, then every gap of $\cA+\cA$ below $n$ is at most $\tau\log n$ once $n>s_{j_0}$, so $\gamma_{\cA+\cA}(n)\le\tau\log n$ for all large $n$, and the first claim gives $\tau\ge\frac{\log 2}{2}$.
\end{proof}

The same tiling argument applies to $\cA$ itself, where it improves the trivial bound by a power of a logarithm. Since $A(n)=O(n^{d/h})$ for every $B_h$-set (Lemma~\ref{lem:easy Bh bound}), the cube $[0,n)^d$ trivially contains an empty cube of volume $\gg n^{d(h-1)/h}$ for every $n$. Corollary~\ref{cor:holes} gains the factor $(\log n)^{1/h}$, but only infinitely often.
\begin{cor}{Large Gaps in $B_h$-sets}{holes}
Let $d\ge1$ be an integer, let $h$ be a positive even integer, and let $\cA\subseteq\NN^d$ be an infinite $B_h$-set. Then
  \[\limsup_{n\to\infty}\frac{\gamma_\cA(n)}{\left(n^{d(h-1)}\log n\right)^{1/h}}\ge\frac1{C_{d,h}}.\]
\end{cor}

\begin{proof}
Suppose, by way of contradiction, that there are $\tau<1/C_{d,h}$ and $n_1$ with $\gamma_\cA(n)\le\tau n^{d(h-1)/h}(\log n)^{1/h}$ for all $n\ge n_1$. Fix such an $n$ and set
  \[g\coloneqq\floor*{\bigl(\tau n^{d(h-1)/h}(\log n)^{1/h}\bigr)^{1/d}}+1,\]
so that $g^d>\gamma_\cA(n)$ and consequently every cube of side $g$ contained in $[0,n)^d$ meets $\cA$. The cube $[0,n)^d$ contains $\floor{n/g}^d$ pairwise disjoint cubes of side $g$, each containing an element of $\cA$. Thus, as $n\to\infty$,
  \begin{align*}
    A(n)\ge\floor*{\frac ng}^d
    &\ge\left(\frac ng-1\right)^d\\
    &=(1-o(1))\frac{n^d}{g^d}\\
    &=(1-o(1))\frac{n^d}{\tau n^{d(h-1)/h}(\log n)^{1/h}}\\
    &=\frac{1-o(1)}{\tau}\left(\frac{n^d}{\log n}\right)^{1/h},
  \end{align*}
using $g^d=(1+o(1))\tau n^{d(h-1)/h}(\log n)^{1/h}$ and $g=o(n)$. We now see that
  \[\liminf_{n\to\infty}\frac{A(n)}{\left(n^d/\log n\right)^{1/h}}\ge\frac1\tau>C_{d,h},\]
contradicting Corollary~\ref{cor:cube}.
\end{proof}

\subsection{History of the problem}\label{sec:history}
Erd\H{o}s proved, as St\"ohr~\cite{1955.Stohr} announced, that every infinite $B_2$-set $\cA\subseteq\NN$ satisfies
\begin{equation*}\label{eq:erdos}
   \liminf_{n\to\infty}\frac{\abs{\cA\cap[0,n)}}{\sqrt{n/\log n}} < C ,\tag{$h=2,d=1$}
\end{equation*}
with an unspecified absolute constant $C$. Erd\H{o}s's proof is nicely exposited in Halberstam~\&~Roth~\cite{1966.Halberstam-Roth}. Erd\H{o}s~\&~Tur\'an~\cite{1941.Erdos&Turan} proved that if $\cA\subseteq[0,n)$ is a $B_2$-set, then $|\cA| \le \sqrt{n}+O(n^{1/4})$. Thus, \emph{infinite} $B_2$-sets must be thinner than finite $B_2$-sets at infinitely many $n$ by at least a power-of-a-logarithm factor.

In a series of papers Chen~\cite{1996.Chen}, Nash~\cite{1989.Nash}, Jia~\cite{1989.Jia,1994.Jia}, Helm~\cite{1993.Helm,1994.Helm}, and finally Chen~\cite{1993.Chen} extended this to every even $h$: for every $B_h$-set $\cA\subseteq\NN$,
  \begin{equation*}
    \liminf_{n\to\infty} \frac{\abs{\cA\cap[0,n)}}{\left(n/\log n\right)^{1/h}} < \infty.\tag{$h$ even, $d=1$}
  \end{equation*}
Chen's 1993 paper was written after but published before several other works in this list.
Jia's~\cite{1994.Jia} argument, whose outline we follow, needed an additional distribution hypothesis $A(n^2)=O(A(n)^2)$. Singer~\cite{1938.Singer} gives $B_h$-sets in $[0,n)$ with $n^{1/h}$ elements for infinitely many $n$, and the obvious combinatorial bound (Lemma~\ref{lem:easy Bh bound}) gives $|\cA|=O(n^{1/h})$. The optimal coefficient remains unknown, if there even is an optimal coefficient.

$B_2$-sets in $\NN^d$ (that is, $h=2$ but $d>1$) were introduced by Lindstr\"om~\cite{1972.Lindstrom}, who showed that a $B_2$-set in $[1,n]^d$ has at most $n^{d/2}+O(n^{d^2/(2d+2)})$ elements. For infinite $B_2$-sets in $\NN^2$, Trujillo~\cite{1998.Trujillo} proved that $\abs{\cA\cap([1,n]\times[1,m])}\ll\sqrt{nm/\log(nm)}$ for some arbitrarily large pairs $(n,m)$. Cilleruelo~\cite[Theorem 6]{2010.Cilleruelo-a} then proved that every infinite $B_2$-set $\cA\subseteq\NN^d$ has
  \begin{equation*}
    \liminf_{n\to\infty}\frac{\abs{\cA\cap[0,n)^d}}{\sqrt{n^d/\log n}} <\infty
    \tag{$h=2,d\ge 1$}
  \end{equation*}
and showed how to transfer growth exponents between dimensions. 

Rackham \& \v{S}arka~\cite{2010.Rackham&Sarka} then combined Cilleruelo's and Chen's arguments to prove that
  \begin{equation*}
    \liminf_{n\to\infty}\frac{A(n)}{\left(n^d/\log n\right)^{1/h}} <\infty
    \tag{$h$ even, $d\ge 1$}
  \end{equation*}
for every $B_h$-set $\cA\subseteq\NN^d$ with $h$ even. The finite problem has a large literature, and the strongest result seems to be that of Johnston \& Tait \& Timmons~\cite{2022.Johnston&Tait&Timmons}.

Cilleruelo~\cite{2010.Cilleruelo-a} has given a constant for $d=1,h=2$, namely $8\sqrt{7}\approx 21.2$.
In Part~I of this series~\cite{2026.Obryant-a}, the author proved the $d=1$, $h=2$ case of Corollary~\ref{cor:cube}, and more generally that if $\cG\subseteq\NN$ is a $g$-Golomb ruler (no $\delta$ arises more than $g$ times as a difference of elements of $\cG$, so that a Sidon set is a $1$-Golomb ruler), then
\begin{equation*}\label{eq:g-Golomb}
  \liminf_{n\to\infty}\frac{G(n)}{\sqrt{n/\log n}} \le
  \sqrt{\frac{4g}{\log 2}}.
\end{equation*}
With $g=1$, this improves Cilleruelo's constant to about $2.4$. There is no published constant for any other $d,h$, prior to this work.

\subsection{An outline of the proof, with emphasis on our contribution}

Our Theorem~\ref{thm:main} improves on prior work by giving an explicit (and small) constant $C_{d,h}$, by applying to subsets of $\ZZ^d$ instead of merely $\NN^d$, and by allowing more general counting functions. 

We note that the constant does not depend on the shape of the region $K$ used to count elements of $\cA$. The finite analogue of this shape-independence is a theorem of Cilleruelo~\cite[Corollary 3]{2010.Cilleruelo-a}: a $B_2$-set contained in the dilate $t\ast C$ of a convex set $C\subseteq(0,1]^d$ has at most $\sqrt{\vol(t\ast C)}+O(t^{d^2/(2d+2)})$ elements. We hope that our explicit constant will spur further work, and even for $d=1$ and $h=2$ we do not expect $C_{1,2}=2/\sqrt{\log 2}$ to be the truth, and likely $\sqrt{n/\log n}$ is not even the right normalizing factor.

\subsubsection{A description of the proof}\label{sec:method}
We now describe the proof at a high level. We begin with a hypothetical $B_{2k}$-set $\cA\subseteq\ZZ^d$ that is thicker than we will prove possible, meaning that $A_K(x)\ge\tau(\vol(x\ast K)/\log x)^{1/h}$ for all large $x$ and some $\tau>C_{d,h}$. We consider $\cS=k\cA$, the $k$-fold sumset (not $h$-fold!). For $\bar y\in\ZZ^d$ we define $G_N(\bar y)$ to be the number of elements of $\cS$ in the cube $\bar y+[0,N)^d$. The $2$-norm of $G_N$ is a sort of ``energy.'' We first show that the energy is small, and then show that if $\cS$ grows quickly then the energy is large. The two bounds conflict, and the constant $C_{d,h}$ is chosen to get a contradiction.

The energy is small because $\cS$ has few repeated differences. Two elements of $\cS$ in the same cube differ by a vector $\bar v$ with $\norm{\bar v}<N$. If their representations as sums of $k$ elements of $\cA$ share no summand, then $\bar v$ determines the pair, because $\bar s-\bar s'=\bar u-\bar u'$ gives $\bar s+\bar u'=\bar u+\bar s'$, an equality of $2k$-fold sums, and $\cA$ is a $B_{2k}$-set. So, apart from pairs whose representations overlap, each lag $\bar v$ is used at most once, and the energy is at most about $N^{2d}$.

The energy is large because of Cauchy's inequality. If $w$ is any weight on $\ZZ^d$, then $\langle G_N,w\rangle^2\le\norm{w}_2^2 \, \norm{G_N}_2^2$. We take $w(\bar y)$ to be a multiple of $\rho(\bar y)^{-d/2}(\log\rho(\bar y))^{-1/2}$ on the window $N\le\rho(\bar y)<W$, the weight that makes Cauchy's inequality tight when $\cS$ has the counting function $\sqrt{\vol(x\ast K)/\log x}$ that we are trying to rule out. A summation by parts turns the lower bound on the counting function of $\cS$ into a lower bound on $\langle G_N,w\rangle$.

This is roughly the argument of Erd\H{o}s, as adapted by Rackham~\&~\v{S}arka, further adapted to handle subsets of $\ZZ^d$ counted by a gauge, and further technical changes to produce an efficient constant $C_{d,h}$. Specifically, we include four ideas, all of which are in Part I of this series in weaker forms.

\begin{idea}{Separate the width of the cubes from the length of the window}{window}
The previous authors take $W=N^2$ (Erd\H os, Chen) or $W=N^{2k/(2k-1)}$ (Rackham \& \v{S}arka). We take $W=N^2/\log^3N$. The weight $w$ has $\norm{w}_2^2\approx d\vol(K)\log(\log W/\log N)$, and $\langle G_N,w\rangle$ has the same logarithm, so the ratio $\log W/\log N\to2$ is what puts the $\log 2$ factor into $C_{d,h}$, while the factor $\log^3N$ is needed to make Idea~\ref{idea:strata} work.
\end{idea}

\begin{idea}{Average over every translate of the cube}{translates}
Rather than partitioning the window into blocks, we sum over every $\bar y\in\ZZ^d$. The energy becomes $\norm{1_\cS*1_{[0,N)^d}}_2^2$, which counts each pair $\bar s,\bar s'$ with weight $\prod_i(N-\abs{s_i-s'_i})_+$. Summing that weight over all lags gives exactly $N^{2d}$, so the pairs with non-overlapping representations contribute at most $\frac12N^{2d}$ to $\sum_{\bar y}\binom{G_N(\bar y)}{2}$, with no loss (Lemma~\ref{lem:sumset energy bound}).
\end{idea}

\begin{idea}{Pairs with overlapping representations are negligible}{strata}
The pairs $\bar s,\bar s'$ whose representations share $r\ge1$ summands are not controlled by the $B_{2k}$-property alone. A lemma of Jia~\cite{1994.Jia} bounds them by $N^d$ times the ``regularity ratio'' $A_K(W)^r/A_K(N)^{2r}$. Jia assumed $A(N^2)=O(A(N)^2)$ to make this ratio bounded, and Chen removed the assumption by a case analysis. We do neither: because $W=N^2/\log^3N$ rather than $N^2$, the trivial upper bound on $A_K(W)$ and the hypothesis on $A_K(N)$ make the ratio tend to zero, so these pairs contribute $o(N^{2d})$ to the energy (Lemma~\ref{lem:strata}).
\end{idea}

\begin{idea}{Count the sumset with a Dirichlet integral}{simplex}
The hypothesis is on $\cA$, but the energy argument needs a lower bound on the counting function of $\cS=k\cA$. The previous authors use $\abs{\cS\cap[0,kn)^d}\ge\binom{A(n)}{k}$, which loses a constant factor. We instead enumerate $\cA$ by gauge, so that the $j$-th element has $\rho(\bar a_j)<x_j$ with $x_j$ explicit, and count the $k$-tuples $(j_1,\dots,j_k)$ with $x_{j_1}+\cdots+x_{j_k}\le m$ by the volume of a region of $\RR^k$. That volume is a Dirichlet integral, which is where the gamma functions in $C_{d,h}$ come from (Lemma~\ref{lem:simplex}).
\end{idea}

The dimension enters the constant only through the ``surface area'' of the region in which we count. The shell $\{\bar y : m\le\rho(\bar y)<m+1\}$ has about $d\vol(K)m^{d-1}$ lattice points, because $\vol(r\ast K)=\vol(K)r^d$ for every $K$. In Cauchy's inequality, $\vol(K)$ cancels between the two sides while the factor $d$ survives. This is why the constant $C_{d,h}$ depends on $d$ but not on $K$.

Passing from $\NN^d$ to $\ZZ^d$ costs little, but one thing that is free in $\NN^d$ must be paid for in $\ZZ^d$. In $\NN^d$, a sum of elements of $\cA$ that lies in $[0,n)^d$ has all of its summands in $[0,n)^d$, and in $\ZZ^d$ nothing of the kind is true. We address this in Section~\ref{sec:differences} by working with the sumsets of finite truncations of $\cA$.

To date, all of the best work on $B_h$-sets, including ours, proceeds by considering the differences of the $(h/2)$-fold sumset of $\cA$, and this is why $h$ must be even. Every $B_h$-set is also a $B_{h-1}$-set, so our result extends to odd $h$, but only in that artificial (and ultimately trivial) manner. The major obstacle to a deeper understanding $B_h$-sets is this dependence of proofs on the differences of $h/2$-fold sums.

\section{Standard notation}

\subsection{Multiset notation and terminology}\label{sec:multisets}
In this work, a multiset $\beta$ (of elements of $\ZZ^d$) is a function $\mult_\beta:\ZZ^d\to \NN$.
We write $\bar x\in \beta$ for $\mult_\beta(\bar x) \ge 1$, and $\multi{\bar a_1,\dots,\bar a_i}$ for the multiset $\beta$ with $\mult_\beta(\bar x) = \abs{\{ j : 1\le j \le i, \bar a_j = \bar x\}}$.
The cardinality and sum of a multiset are
\begin{align*}
  |\beta| & \coloneqq \sum_{\bar x\in\ZZ^d} \mult_\beta(\bar x) \\
  \Sigma \beta & \coloneqq \sum_{\bar x\in \ZZ^d} \mult_\beta(\bar x) \cdot \bar x \in \ZZ^d.
\end{align*}
The multiset $\emptyset$ has $\mult_\emptyset(\bar x) = 0$ for all $\bar x$.

The support of a multiset $\beta$ is the \emph{set} of vectors with multiplicity at least $1$:
  \[\supp(\beta) \coloneqq \{ \bar a : \mult_\beta(\bar a)\ge 1\}.\]
We define the multiset intersection $\beta \Cap \delta$ and sum $\beta \uplus \delta$ through the $\mult$ function:
\begin{align*}
  \bar x \in \beta &\qquad \Leftrightarrow\qquad \mult_\beta(\bar x) \ge 1 \\
  U = \beta \Cap \delta &\qquad  \Leftrightarrow\qquad  \mult_U(\bar x) = \min\{\mult_\beta(\bar x),\mult_\delta(\bar x)\} \\
  U = \beta \uplus \delta &\qquad  \Leftrightarrow\qquad  \mult_U(\bar x) = \mult_\beta(\bar x)+\mult_\delta(\bar x)
\end{align*}
The usual set operations $\cup, \cap$ apply to the supports of their arguments, e.g., $\beta \cap \delta \coloneqq \supp(\beta) \cap \supp(\delta)$. We say that $\beta,\delta$ are \emph{disjoint} if $\beta\cap\delta = \emptyset$.

The set of all multisubsets with support contained in the set $X$ with cardinality $i$ is denoted $\multibinom{\!X\!}{\!i\!}$ (read ``$X$ multichoose $i$''), and by stars-and-bars we have the count
  \[\abs*{\multibinom{X}{i}}= \multibinom{\abs{X}}{i} = \binom{|X|+i-1}{i}.\]
We will use both sides of the standard bounds
  \[ \frac{|X|^i}{i!}\le \abs*{\multibinom{X}{i}} \le \frac{(|X|+i)^i}{i!},\]
and, for $i\ge1$, also the usually cruder $\abs*{\multibinom{X}{i}}\le\abs{X}^i$.

This next lemma is used frequently to connect sumsets and difference sets of a $B_h$-set. It has the same essence as the ubiquitous fact from elementary number theory: if $p,q$ are positive integers and $\gcd(p,q)=\gcd(p',q')=1$ and $pq'=p'q$, then $p=p'$ and $q=q'$.
\begin{lem}{Cancellation Lemma}{cancel}
Let $P,P',Q,Q'$ be multisets with $P\cap Q=P'\cap Q'=\emptyset$.  If $P\uplus Q'=P'\uplus Q$, then $P=P'$ and $Q=Q'$.
\end{lem}
\begin{proof}
Fix $\bar x$. If $\mult_P(\bar x)\ge1$ then $\mult_Q(\bar x)=0$, so $\mult_{P'}(\bar x)+\mult_Q(\bar x)=\mult_P(\bar x)+\mult_{Q'}(\bar x)$ gives $\mult_{P'}(\bar x)\ge\mult_P(\bar x)\ge1$, whence $\mult_{Q'}(\bar x)=0$ and $\mult_{P'}(\bar x)=\mult_P(\bar x)$. Symmetrically if $\mult_{P'}(\bar x)\ge1$. If $\mult_P(\bar x)=\mult_{P'}(\bar x)=0$ then $\mult_{Q'}(\bar x)=\mult_Q(\bar x)$. Thus $P=P'$, and then $Q=Q'$.
\end{proof}

A quick example using multisets is enlightening, albeit ornately pedantic.
\begin{lem}{Monotonicity Lemma}{B_h is B_p}
  Suppose that $\cA\subseteq\ZZ^d$ is a $B_h$-set and $1\le p \le h$. Then $\cA$ is a $B_p$-set.
\end{lem}
\begin{proof}
If $\cA=\emptyset$ or $p=h$ or $p=1$, then the lemma is trivial, but true.

Otherwise, suppose by way of contradiction that $\bar a\in\cA$ and $\beta,\delta \in \multibinom{\cA}{p}$ with $\beta\neq \delta$, yet $\Sigma \beta = \Sigma \delta$. Let $\alpha$ be the multiset containing only $\bar a$, and containing it with multiplicity $h-p$. Then
  \[\Sigma (\alpha \uplus \beta) = \Sigma(\alpha \uplus \delta)\qquad \text{and}\qquad
  \qquad |\alpha \uplus \beta| = |\alpha \uplus \delta| =h.\]
By the $B_h$-property, $\alpha \uplus \beta = \alpha \uplus \delta$, and the Cancellation Lemma~\ref{lem:cancel} gives $\beta =\delta$, contradicting our assumption that $\beta\neq \delta$.
\end{proof}

If $\cA$ is a $B_h$-set, and $\beta,\delta \in \multibinom{\cA}{h}$ and $\Sigma \beta = \Sigma \delta$, then $\beta=\delta$.
This is just a restatement of the $B_h$-property.
In other words, by the $B_h$-property (which implies the $B_p$ property for $1\le p\le h$), the elements of the $p$-fold sumset $p\cA\subseteq\ZZ^d$ are in bijective correspondence with $\multibinom{\cA}{p}$.

\begin{cor}{Each Sum Gives a Unique Multiset}{Vs}
Let $\cA\subseteq\ZZ^d$ be a $B_h$-set and $1\le p\le h$. The map $\beta\mapsto\Sigma\beta$
is a bijection from $\multibinom{\cA}{p}$ onto the $p$-fold sumset $p\cA$.
\end{cor}

For $\bar s\in p\cA$ we write $\lambda_{\bar s}$ for the unique preimage (specifying $p$ in context), so $\Sigma \lambda_{\bar s}=\bar s$.

\subsection{Convex bodies, gauges, and lattice points}\label{sec:convex}
Except where explicitly noted otherwise, $K\subseteq\RR^d$ is a convex body with $\bar 0$ in its \emph{interior}. The case $\bar 0\in\partial K$ is handled at the end of Section~\ref{sec:main proof} by replacing $K$ with an $\epsilon$-fattened $K$ and taking $\epsilon\to0$. We write 
  \[\rho\coloneqq\rho_K \qquad\text{and}\qquad V\coloneqq\vol(K).\] 
Since $\bar 0$ is interior, $\rho$ is finite everywhere, and the gauge is positively homogeneous and subadditive: For $t\ge 0$
  \[\rho(t\bar x)=t\rho(\bar x) \qquad\text{and}\qquad \rho(\bar x+\bar y)\le\rho(\bar x)+\rho(\bar y),\]
the latter being exactly the convexity of $K$. Moreover $\{\bar x \in \RR^d:\rho(\bar x)\le r\}=r\ast K$ for $r>0$, so that $\vol\{\bar x : \rho(\bar x)\le r\}=Vr^d$. As $\rho$ is positive on the compact set $\partial K = \{\bar x : \norm{\bar x}=1\}$, there is a constant $c_K\ge1$ with
  \begin{equation}\label{eq:gauge vs norm}
    \frac{1}{c_K}\norm{\bar x}\le\rho(\bar x)\le c_K\norm{\bar x}\qquad(\bar x\in\RR^d).
  \end{equation}
In particular $[-1,1]^d\subseteq c_K\ast K$, and $\rho(-\bar x)\le c_K\norm{\bar x}\le c_K^2\rho(\bar x)$. For $x>0$ and $\cA\subseteq\ZZ^d$ we write the truncation and counting functions as
  \[\cA_x\coloneqq\{\bar a\in\cA:\rho(\bar a)<x\}\qquad\text{and}\qquad A_K(x)=\abs{\cA_x}.\]

\begin{lem}{Lattice points}{lattice}
For $r\ge c_K$,
  \[V(r-c_K)^d\le\#\{\bar z\in\ZZ^d:\rho(\bar z)\le r\}\le V(r+c_K)^d.\]
\end{lem}
\begin{proof}
The cubes $\bar z+[0,1)^d$, $\bar z\in\ZZ^d$, are pairwise disjoint. If $\rho(\bar z)\le r$ and $\bar t\in[0,1)^d$, then $\rho(\bar z+\bar t)\le\rho(\bar z)+\rho(\bar t)\le r+c_K$, so the cubes with $\rho(\bar z)\le r$ lie in $\{\bar x\in\RR^d: \rho(\bar x)\le r+c_K\}$, whose volume is $V(r+c_K)^d$. This is the upper bound. For the lower bound, let $\rho(\bar x)\le r-c_K$ and let $\bar z\in\ZZ^d$ be the coordinatewise floor of $\bar x$, so that $\bar z-\bar x\in(-1,0]^d$ and $\rho(\bar z)\le\rho(\bar x)+\rho(\bar z-\bar x)\le r$. Thus $\{\bar x \in \RR^d : \rho(\bar x) \le r-c_K\}$ is covered by the cubes with $\rho(\bar z)\le r$, and the lower bound follows by comparing volumes.
\end{proof}

\begin{lem}{Polar coordinates}{polar}
For every measurable $F:[0,\infty)\to[0,\infty]$,
    \[\int_{\RR^d}F(\rho(\bar x))\dd\bar x= (V \cdot d) \int_0^\infty F(r)r^{d-1}\dd r.\]
\end{lem}
\begin{proof}
The pushforward of Lebesgue measure under $\rho$ is the measure $\mu$ on $[0,\infty)$ with $\mu([0,r])=\vol\{\bar x \in \RR^d: \rho(\bar x) \le r\}=Vr^d$, which has density $Vd r^{d-1}$.
\end{proof}

We begin with the simple bound on the size of a $B_h$-set in $\{\rho<m\}$. There is a large literature around improving the constant in our next lemma. For the purposes of this work, surprisingly, Lemma~\ref{lem:easy Bh bound} suffices.

\begin{lem}{Easy Bound on $B_h$-sets}{easy Bh bound}
  If $\cG\subseteq\ZZ^d$ is a $B_h$-set and $m\ge 1$, then $G_K(m) \le (h!V)^{1/h}(hm+c_K)^{d/h}$. In particular, $G_K(m)=O(m^{d/h})$.
\end{lem}

\begin{proof}
  There are $G_K(m)^h$ ordered $h$-tuples of elements of $\cG_m$. Each tuple has a sum $\bar s$ with $\rho(\bar s)<hm$ by subadditivity, there are at most $V(hm+c_K)^d$ such lattice points by Lemma~\ref{lem:lattice}, and by the $B_h$-property each lattice point is the sum of at most $h!$ tuples. Thus $G_K(m)^h\le h!V(hm+c_K)^d$.
\end{proof}

\section{Lemmata on \texorpdfstring{$B_h$}{B sub h}-sets in \texorpdfstring{$\ZZ^d$}{Z to the d}}\label{sec:lemmata}
We continue to assume that $\bar 0$ is interior to $K$. We will apply the lemmata of this section to a $B_h$-set $\cA\subseteq\ZZ^d$, $h=2k$, satisfying the
\begin{dfn}{Growth Hypothesis for $\tau$}{growth}
There are $\tau>0$ and $n_0\ge 3$ such that for all real $x\ge n_0$
  \[
    A_K(x)\ge\tau\left(\frac{Vx^d}{\log x}\right)^{1/h}.
  \]
\end{dfn}
This is the negation of the conclusion of Theorem~\ref{thm:main} for $\tau>C_{d,h}$, and our proof of Theorem~\ref{thm:main} is by contradiction. The constants implied by $O$ and $o$ may depend on $d$, $h$, $K$, $\tau$, and $n_0$, but on nothing else.

Our summation juggling is aided by the Iverson bracket.
\begin{dfn}{Iverson Bracket}{tf notation}
  Given a statement $P$ and formula $X$, the notation $\tf{P}X$ is $0$ if $P$ is false and is $X$ if $P$ is true. Note that $X$ need only be defined when $P$ is true.
\end{dfn}

\subsection{Energy and the counting function}
Our first lemma connects an ``energy'' quantity to the counting function. It is stated for a family of finite sets $\cB_N$, one for each $N$, rather than for a single infinite set; the reason is explained in Section~\ref{sec:differences}. For a finite $\cB\subseteq\ZZ^d$, $N\ge 2$, and $\bar y\in\ZZ^d$ we define the \emph{cube count}
  \begin{equation}\label{eq:G def}
    G_N(\bar y) \coloneqq \abs*{\cB\cap (\bar y+[0,N)^d)} = \bigl(1_\cB * 1_{(-N,0]^d}\bigr)(\bar y) = \sum_{\bar b \in \cB} \tf{\bar b \in \bar y + [0,N)^d}.
  \end{equation}

For $d=1$, $K=[0,1]$, and $\cB\subseteq\NN$, summing over $\bar y$ is the same as summing over the $N$ offsets $t\in[0,N)$ and the blocks $[t+(\ell-1)N,t+\ell N)$, and so the following lemma is (in that case) a mild variant of~\cite[Lemma 6]{2026.Obryant-a}, with the energy averaged over offsets rather than taken at the best offset. We give the proof in full, as the dependence on $d$ and $K$ was not visible before.

\begin{lem}{Energy Lemma}{energy}
Let $W=W(N)\ge N$ satisfy $\frac{\log W}{\log N} \to2$ as $N\to\infty$. For each $N$, let $\cB_N\subseteq\ZZ^d$ be a finite set, with counting function $B_N(m)\coloneqq\#\{\bar b\in\cB_N:\rho(\bar b)<m\}$ and cube counts $G_N$ as in~\eqref{eq:G def}. Suppose that, as $N\to\infty$,
\begin{enumerate}[label=\textit{(\roman*)},noitemsep]
  \item\label{eng:window} $\abs{\cB_N}=o(N^d)$;
  \item\label{eng:head} $B_N((2+c_K)N)=o(\sqrt{N^d \log N})$;
  \item\label{eng:energy} there is a constant $c$ with $\displaystyle\sum_{\bar y\in\ZZ^d}\binom{G_N(\bar y)}{2}\le cN^{2d}+o(N^{2d})$.
\end{enumerate}
If $\tau>0$ is such that, for all sufficiently large $N$ and all $m$ with $N\log N\le m\le W$, we have
  \begin{equation}\label{hyp:profile}
    B_N(m)\ge\tau\sqrt{\frac{Vm^d}{\log m}},
  \end{equation}
then $\tau\le\sqrt{8c/(d\log 2)}$.
\end{lem}

\begin{proof}
Throughout, $N$ is large.
Let $\phi(x)\coloneqq x^{-d/2}(\log x)^{-1/2}$ for $x\ge 2$; this is positive and decreasing, with
  \begin{equation}\label{eq:phi'}
    -\phi'(x) = \phi(x)\left(\frac{d}{2x}+\frac{1}{2x\log x}\right) \ge \frac{d}{2x}\phi(x).
  \end{equation}
Define the weight $w\colon\ZZ^d\to[0,\infty)$ using the Iverson notation as 
  \[w(\bar y)\coloneqq \tf{N\le\rho(\bar y)<W}\, \phi(\rho(\bar y)),
  \] 
We will bound the three quantities in Cauchy's inequality 
  \[\langle G_N,w\rangle^2 \le \norm{w}_2^2 \cdot \sum_{\bar y}G_N(\bar y)^2.\]

\begin{claim}
\label{claim:norm w}
  \(\displaystyle
    \norm{w}_2^2\le Vd \log 2+o(1).
  \)
\end{claim}

\begin{proof}[Proof of Claim~\ref{claim:norm w}]
Let $\bar y\in\ZZ^d$ with $N\le\rho(\bar y)<W$ and let $\bar x\in\bar y+[0,1)^d$. By~\eqref{eq:gauge vs norm}, we have $\rho(\bar x)-c_K \le \rho(\bar y) \le \rho(\bar x)+c_K$ and $N-c_K\le\rho(\bar x)<W+c_K$. By the monotonicity of $w$, we have $0\le w(\bar y) \le \phi(\rho(\bar x)-c_K)$.

The cubes $\bar y+[0,1)^d$ are pairwise disjoint, so by Lemma~\ref{lem:polar} (Polar Coordinates)
  \begin{align*}
    \norm{w}_2^2 \coloneqq \sum_{\bar y \in \ZZ^d} w(\bar y)^2 
    &\le\int_{N-c_K\le\rho(\bar x)<W+c_K}\phi(\rho(\bar x)-c_K)^2d\bar x\\
    &= Vd \int_{N-c_K}^{W+c_K}\frac{r^{d-1}dr}{(r-c_K)^d\log(r-c_K)}\\
    &= Vd \int_{N-2c_K}^{W}\Bigl(1+\frac{c_K}{s}\Bigr)^{d-1}\frac{ds}{s\log s}.
  \end{align*}
For $N-2c_K\le s<N\log N$ we use $(1+c_K/s)^{d-1}\le 2^{d-1}$ and
  \[\int_{N-2c_K}^{N\log N}\frac{ds}{s\log s}=\log\frac{\log(N\log N)}{\log(N-2c_K)} = o(1).\]
For $s\ge N\log N$ we use $(1+c_K/s)^{d-1}\le(1+1/\log N)^{d-1}=1+o(1)$ and
  \[\int_{N\log N}^{W}\frac{ds}{s\log s}= \log\frac{\log W}{\log(N\log N)}=\log 2+o(1).\]
Together,
  \(\norm{w}_2^2\le Vd \log 2+o(1),\) as claimed.
\end{proof}

\begin{claim}\label{claim:inner product}
  \(\displaystyle 
    \langle G_N,w\rangle\ge N^d\left(\frac{d\tau\sqrt V\log 2}{2}-o(1)\right).
  \)
\end{claim} 

\begin{proof}[Proof of Claim~\ref{claim:inner product}]
Exchanging the order of summation in the definition of $G_N$,
  \begin{align*}
  \langle G_N,w\rangle 
   \coloneqq \sum_{\bar y\in\ZZ^d}G_N(\bar y)w(\bar y)  
  & = \sum_{\bar y\in\ZZ^d} \left( \sum_{\bar b \in \cB_N} \tf{\bar b \in \bar y + [0,N)^d} \right) \cdot w(\bar y) \\
  &= \sum_{\bar b \in \cB_N} \sum_{\bar y \in \ZZ^d}\tf{\bar y \in \bar b- [0,N)^d}\tf{N\le \rho(\bar y) < W} \, \phi(\rho(\bar y)).
  \end{align*}
We restrict our attention to those $\bar b\in\cB_N$ with $(1+c_K)N\le\rho(\bar b)<W-c_KN$, and let $\bar y=\bar b-\bar t$ with $\bar t\in[0,N)^d$. Then 
  \[
  N \le \rho(\bar b)-\rho(\bar t) \le \rho(\bar y) \le \rho(\bar b)+\rho(-\bar t)<W.\]
Hence $\tf{N\le \rho(\bar y) < W}=1$. Moreover, for such a $\bar b$ we have $\phi(\rho(\bar y))\ge\phi(\rho(\bar b)+c_KN)$. Each such $\bar b$ therefore contributes at least $N^d\phi(\rho(\bar b)+c_KN)$. We have
  \[
  \langle G_N,w\rangle \ge \sum_{\bar b \in \cB_N} \tf{(1+c_K)N\le\rho(\bar b)<W-c_KN} N^d \phi(\rho(\bar b)+ c_K N).\]
Put $m_1\coloneqq\ceiling{(1+c_K)N}$, $m_2\coloneqq\floor{W-c_KN}$, and $\psi(m)\coloneqq\phi(m+1+c_KN)$, a decreasing function of $m$. An element $\bar b$ with $m\le\rho(\bar b)<m+1$ contributes at least $N^d\psi(m)$, and there are $B_N(m+1)-B_N(m)$ of them, so
  \[\langle G_N,w\rangle \ge N^d\sum_{m=m_1}^{m_2-1}\psi(m)\bigl(B_N(m+1)-B_N(m)\bigr).\]
Summation by parts gives
  \begin{align*}
    \sum_{m=m_1}^{m_2-1}\psi(m)\bigl(B_N(m+1)-B_N(m)\bigr)
    &= \psi(m_2-1)B_N(m_2)-\psi(m_1)B_N(m_1)\\
    &\qquad+\sum_{m=m_1+1}^{m_2-1}\bigl(\psi(m-1)-\psi(m)\bigr)B_N(m).
  \end{align*}
The first term is nonnegative. For the second, $\psi(m_1)\le\phi((1+c_K)N)\le N^{-d/2}(\log N)^{-1/2}$ and $B_N(m_1)\le B_N((2+c_K)N)$, so the second term is $o(1)$ by hypothesis~\ref{eng:head} of Lemma~\ref{lem:energy}. In the sum, every term is nonnegative, so we may discard those with $m<N\log N$. For $m\ge N\log N$, by the mean value theorem, \eqref{eq:phi'}, and monotonicity, for some $\xi \in (m+c_KN,m+1+c_KN)$ 
  \[
    \psi(m-1)-\psi(m) =-\phi'(\xi) 
    \ge\frac{d}{2(m+1+c_KN)}\phi(m+1+c_KN).
  \]
Because $m\ge N \log N$, we have
  \[\frac{\phi(m+1+c_KN)}{m+1+c_KN} \cdot \frac{m}{\phi(m)} =1+o(1),\]
and so
  \[ \psi(m-1)-\psi(m) \ge(1-o(1))\frac{d}{2m}\phi(m),\]
Now apply~\eqref{hyp:profile}, noting that $\phi(m)\sqrt{Vm^d/\log m}=\sqrt V/\log m$, and we see that
  \begin{align*}
    \sum_{m=\ceiling{N\log N}}^{m_2-1}\bigl(\psi(m-1)-\psi(m)\bigr)B_N(m)
    &\ge (1-o(1))\frac{d\tau\sqrt V}{2}\sum_{m=\ceiling{N\log N}}^{m_2-1}\frac{1}{m\log m}\\
    &\ge(1-o(1))\frac{d\tau\sqrt V}{2}\bigl(\log 2-o(1)\bigr),
  \end{align*}
because the sum is at least $\log\log m_2-\log\log\ceiling{N\log N}$, and $\log m_2=\log W+O(1)$. Thus
  \(
    \langle G_N,w\rangle\ge N^d\left(\frac{d\tau\sqrt V\log 2}{2}-o(1)\right)
  \) as claimed.
\end{proof}

\begin{claim}\label{claim:energy bound}
  \(\displaystyle \sum_{\bar y\in\ZZ^d}G_N(\bar y)^2 \le 2cN^{2d}+o(N^{2d}).\)
\end{claim}

\begin{proof}[Proof of Claim~\ref{claim:energy bound}.] Since $G_N(\bar y)^2=2\binom{G_N(\bar y)}{2}+G_N(\bar y)$, and since each $\bar b\in\cB_N$ lies in $\bar y+[0,N)^d$ for exactly $N^d$ values of $\bar y$,
  \begin{align*}
    \sum_{\bar y\in\ZZ^d}G_N(\bar y)^2 &\le 2cN^{2d}+o(N^{2d})+N^d\abs{\cB_N}\\
    &= 2cN^{2d}+o(N^{2d})
  \end{align*}
by~\ref{eng:energy} and~\ref{eng:window}.
\end{proof}

Plugging Claims~\ref{claim:norm w},~\ref{claim:inner product}, and~\ref{claim:energy bound} into Cauchy's inequality gives
  \[
  N^{2d}\left(\frac{d\tau\sqrt V\log 2}{2}-o(1)\right)^2
  \le \bigl(Vd\log 2+o(1)\bigr) \cdot \bigl(2cN^{2d}+o(N^{2d})\bigr).
  \]
Divide through by $N^{2d}$ and let $N\to\infty$, and we see that
  \[\frac{d^2 \tau^2 V \log^2 2}{4} \le Vd \log 2 \cdot 2c ,\]
from which it follows that $\tau\le\sqrt{8c/(d\log 2)}$.
\end{proof}

\subsection{The size of sumsets}
Let $h=2k$, let $\cA\subseteq\ZZ^d$ be a $B_h$-set, and for $m>0$ let
  \begin{equation}\label{eq:S(m) def}
    \cS^{(m)}\coloneqq\bigl(k\cA_m\bigr)\cap\{\bar x \in\RR^d: \rho(\bar x) <m\}
  \end{equation}
be the set of sums of $k$ elements of $\cA_m=\{\bar a \in \cA: \rho(\bar a) <m\}$, for which the sum also has $\rho<m$. For $K=[0,1]^d$ and $\cA\subseteq\NN^d$, $\cS^{(m)}$ is simply $(k\cA)\cap[0,m)^d$. We have the immediate bound
  \begin{equation}\label{eq:S upper}
    \abs*{\cS^{(m)}} \le \abs*{\multibinom{\cA_m}{k}} \le \frac{(A_K(m)+k)^k}{k!} = O(m^{kd/h})=O(m^{d/2})
  \end{equation}
by Corollary~\ref{cor:Vs} and Lemma~\ref{lem:easy Bh bound}, and the immediate lower bound 
  \[\abs{\cS^{(km)}}\ge\abs*{\multibinom{\cA_m}{k}}\ge A_K(m)^k/k!.\] 
We work substantially harder to produce a better lower bound, under the growth hypothesis.
\begin{lem}{Lower bound on sumsets}{simplex}
Let $h=2k$ be even and let $\cA\subseteq\ZZ^d$ be a $B_h$-set satisfying the growth hypothesis~\eqref{dfn:growth} for $\tau$. Then, as $m\to\infty$,
  \[\abs*{\cS^{(m)}} \ge (1-o(1))\frac{\tau^k \Gamma(1+\tfrac dh)^k}{k!\Gamma(1+\tfrac d2)}\sqrt{\frac{Vm^d}{\log m}}.\]
\end{lem}

\begin{proof}
Set $\lambda\coloneqq\tau V^{1/h}$ and $f(x) \coloneqq \lambda(x^d/\log x)^{1/h}$, so that the growth hypothesis reads $A_K(x)\ge f(x)$ for $x\ge n_0$. Since
  \[\frac{f'(x)}{f(x)} = \frac{1}{hx}\left(d-\frac{1}{\log x}\right),\]
$f$ is continuous, strictly increasing, and unbounded on $[e,\infty)$, and $n_0\ge3>e$. In particular $\cA$ is infinite. By Lemma~\ref{lem:lattice} each set $\cA_x$ is finite, so we may enumerate $\cA$ ordered by $\rho$ as
  \[\cA = \{\bar a_1,\bar a_2,\dots\},\qquad \rho(\bar a_1)\le\rho(\bar a_2)\le\cdots.\]
For integers $j\ge J_0\coloneqq \floor{f(n_0)}+1$, let $x_j>n_0$ be the unique solution to $f(x_j)=j$. The schematic in Figure~\ref{fig:x and y} illustrates these relationships.

\begin{figure}[t]
\begin{center}
\resizebox{\textwidth}{!}{%
\begin{tikzpicture}[>={Stealth[length=2.5mm]}, x=0.82cm, y=0.92cm]

  \tikzset{
    closed/.style={red!70!black, fill=red!70!black},
    open/.style  ={red!70!black, fill=white, line width=0.6pt}}

  \draw[->] (-0.2,0) -- (10.4,0) node[right] {$x$};
  \draw[->] (0,-0.2) -- (0,5.7) node[above] {$y$};

  \def\aj{3.0}    
  \def\xj{6.25}   
  \def\jj{4}      

  \draw[dashed] (0,\jj) -- (\xj,\jj);
  \draw[dashed] (\aj,\jj) -- (\aj,0);
  \draw[dashed] (\xj,\jj) -- (\xj,0);
  \node[left] at (0,\jj) {$j$};
  \fill (0,\jj) circle (1.1pt);

  \node[left] at (0,3) {$J_0$};
  \fill (0,3) circle (1.1pt);

  \draw[dotted] (2.0,0) -- (2.0,3.2);
  \node[below, font=\footnotesize] at (2.0,-0.08) {$n_0$};

  \draw[very thick, blue!65!black, domain=0.05:9.6, samples=140, smooth]
        plot (\x, {1.6*sqrt(\x)});
  \node[blue!65!black, anchor=west] at (9.0,4.55) {$y=f(x)$};

  \def\tread#1#2#3{
    \draw[very thick, red!70!black] (#1,#3) -- (#2,#3);
    \draw[open]   (#1,#3) circle (1.7pt);
    \fill[closed] (#2,#3) circle (1.7pt);}
  \draw[very thick, red!70!black] (0,0) -- (0.5,0);
  \fill[closed] (0.5,0) circle (1.7pt);
  \tread{0.5}{1.3}{1}   
  \tread{1.3}{2.0}{2}   
  \tread{2.0}{3.0}{3}   
  \tread{3.0}{5.8}{4}   
  \draw[very thick, red!70!black] (5.8,5) -- (9.6,5);
  \draw[open] (5.8,5) circle (1.7pt);
  \node[red!70!black, anchor=west] at (6.5,5.25) {$y=A_K(x)$};

  \draw[->, gray!55!black] (2.0,5.6) -- (9.6,5.6);
  \node[gray!45!black, font=\footnotesize, anchor=south] at (5.7,5.6)
        {$A_K(x)\ge f(x)\ \ (x\ge n_0)$};

  \fill[blue!65!black] (\xj,\jj) circle (2pt);  
  \node[blue!65!black, anchor=west, font=\footnotesize]
        at (\xj+0.15,\jj-0.25) {$x_j=f^{-1}(j)$};

  \node[below] at (\aj,-0.05) {$\rho(\bar a_j)$};
  \node[below] at (\xj,-0.05) {$x_j$};
  \draw[decorate,decoration={brace,amplitude=4pt,mirror}]
        (\aj,-0.6) -- (\xj,-0.6)
        node[midway,below=3pt,font=\footnotesize] {$\rho(\bar a_j)<x_j$};

  \node[align=center, font=\footnotesize, gray!40!black]
        at (8.4,2.0)
        {$A_K(x_j)\ge f(x_j)=j$\\[1pt]$\Rightarrow\ \rho(\bar a_j)<x_j$};
  \end{tikzpicture}}
\end{center}
\caption{Schematic showing the relationships $A_K(x)\ge f(x)$ for $x\ge n_0$, $f(x_j)=j$, and $\rho(\bar a_j) < x_j$. The staircase $A_K(x)$ counts the points of $\cA$ in $\{\rho<x\}$, so it is a function of the single variable $x$ whatever the dimension and whatever the body $K$.}
\label{fig:x and y}
\end{figure}

Since $A_K(x_j) \ge f(x_j)=j$ for $j\ge J_0$, at least $j$ elements of $\cA$ have $\rho<x_j$, and we see that \(\rho(\bar a_j) < x_j\) for \(j\ge J_0\).
Suppose $J_0\le j_1\le\cdots\le j_k$ are integers with $x_{j_1}+\cdots+x_{j_k}\le m$. Each $\bar a_{j_i}$ has $\rho(\bar a_{j_i})<x_{j_i}\le m$, so $\beta\coloneqq \multi{\bar a_{j_1},\dots,\bar a_{j_k}}\in\multibinom{\cA_m}{k}$, and by subadditivity $\rho(\Sigma\beta)\le\sum_i \rho(\bar a_{j_i})< \sum_i x_{j_i} \le m$, so that $\Sigma\beta\in\cS^{(m)}$. Distinct $\beta$ produce distinct elements of $\cS^{(m)}$ by Corollary~\ref{cor:Vs}.
Hence
  \begin{align*}
    \abs*{\cS^{(m)}}
    &\ge \#\left\{(j_1,\dots,j_k) : J_0\le j_1\le\cdots\le j_k,
        \sum_{i=1}^k x_{j_i}\le m\right\} \\
    &\ge \frac{1}{k!} \cdot \#\left\{(j_1,\dots,j_k)\in\{J_0,J_0+1,\dots\}^k :
    \sum_{i=1}^k x_{j_i}\le m\right\}.
  \end{align*}
We define
  \[  N^*(m) \coloneqq \#\left\{(j_1,\dots,j_k)\in\{J_0,J_0+1,\dots\}^k :
    \sum_{i=1}^k x_{j_i}\le m\right\}.
  \]
We have transformed lower bounding $\abs{\cS^{(m)}}$, a number theory problem, into counting the lattice points in a region of $\RR^k$. Unsurprisingly, we proceed by replacing the lattice point count with an integral and will arrive at
  \begin{equation}\label{eq:N* bound}
    N^{*}(m) \ge \bigl(1-o(1)\bigr) \frac{\lambda^k\Gamma(1+\tfrac dh)^k}{\Gamma(1+\tfrac d2)}\sqrt{\frac{m^d}{\log m}}
  \end{equation}
as $m\to\infty$, from which Lemma~\ref{lem:simplex} follows since $\lambda^k=\tau^kV^{k/h}=\tau^k\sqrt V$.

Each integer tuple $(j_1,\dots,j_k)$ corresponds to the unit box $\prod_{i}(j_i-1,j_i]$, and on that box $\lceil y_i\rceil=j_i$.
Therefore
  \[N^{*}(m) = \vol\left\{(y_1,\dots,y_k) \in (J_0-1,\infty)^k :
     \sum_{i=1}^k x_{\lceil y_i\rceil}\le m\right\}.\]
The function $f$ has an inverse on $[e,\infty)$ that is also continuous and strictly increasing.
For $\ceiling{y_i}\ge J_0$, we have $f^{-1}(y_i+1)\ge f^{-1}(\ceiling{y_i}) = x_{\ceiling{y_i}}$.
Hence
  \begin{multline*}
    \left\{ (y_1,\dots,y_k) \in (J_0-1,\infty)^k : y_i > J_0,
     \sum_{i=1}^k f^{-1}(y_i+1) \le m\right\}\\
     \subseteq
     \left\{(y_1,\dots,y_k) \in (J_0-1,\infty)^k :
     \sum_{i=1}^k x_{\lceil y_i\rceil}\le m\right\}.
  \end{multline*}
The substitution $u_i\coloneqq f^{-1}(y_i+1)$, with $dy_i=f'(u_i)du_i$ and
$u_0\coloneqq \max\{\log m, f^{-1}(J_0+1)\}$, gives
  \[N^{*}(m) \ge \int_{\substack{u_i>u_0\\ u_1+\cdots+u_k\le m}}
     \prod_{i=1}^k f'(u_i)du_i.\]

By calculus, we have for $u_0<u\le m$
  \begin{align*}
    f'(u)
    &= \frac{\lambda d}{h} \frac{u^{d/h-1}}{(\log u)^{1/h}} \left( 1-\frac{1}{d\log u} \right) \\
    &\ge \frac{\lambda d}{h(\log m)^{1/h}} \left( 1-\frac{1}{d\log u_0} \right) u^{d/h-1},
  \end{align*}
and $f'>0$.
Thus,
  \[N^{*}(m) \ge \frac{(\lambda d)^k}{h^k(\log m)^{k/h}} \left( 1-\frac{1}{d\log u_0} \right)^k
  \int_{\substack{u_i>u_0\\ u_1+\cdots+u_k\le m}}
     \prod_{i=1}^k u_i^{d/h-1} du_i.\]

Dirichlet's integral technique~\cite[page 258--9]{1927.Whittaker-Watson} evaluates the untruncated version:
  \[\int_{\substack{u_i>0\\ u_1+\cdots+u_k\le m}}
     \prod_{i=1}^k u_i^{d/h-1}du_i
     = \frac{\Gamma(\tfrac dh)^k}{\Gamma(1+\tfrac{kd}{h})}m^{kd/h}
     = \frac{\Gamma(\tfrac dh)^k}{\Gamma(1+\tfrac d2)}m^{d/2}.\]
The truncation costs little: the portion of the untruncated integral with $u_1\le u_0$ is at most
  \[\int_0^{u_0}u^{d/h-1}du\cdot
    \left(\int_0^m u^{d/h-1}du\right)^{k-1}
    = \left(\frac hd\right)^k u_0^{d/h} m^{(k-1)d/h} = o(m^{d/2}),\]
as $u_0=O(\log m)$, and by symmetry the total cost is at most $k$ times this.
Using $k/h=1/2$, $\frac dh\Gamma(\frac dh)=\Gamma(1+\frac dh)$, $u_0\ge\log m$, and $(1-\epsilon)^k\ge 1-k\epsilon$, we now have
  \begin{align*}
    N^{*}(m)
    &\ge \frac{(\lambda d)^k}{h^k(\log m)^{1/2}} \left( 1-\frac{1}{d\log \log m} \right)^k
  \left(\frac{\Gamma(\tfrac dh)^k}{\Gamma(1+\tfrac d2)}m^{d/2} - o(m^{d/2})\right)\\
    &\ge \frac{\lambda^k \Gamma(1+\tfrac dh)^k}{\Gamma(1+\tfrac d2)} \frac{m^{d/2}}{\sqrt{\log m}} \left( 1-\frac{k}{d\log \log m} \right)
  \left(1 - o(1)\right)\\
  &= (1-o(1)) \frac{\lambda^k \Gamma(1+\tfrac dh)^k}{\Gamma(1+\tfrac d2)} \sqrt{\frac{m^d}{\log m}}.
  \end{align*}
This is~\eqref{eq:N* bound}, and dividing by $k!$ gives the bound reported in Lemma~\ref{lem:simplex}.
\end{proof}

\subsection{Differences of the sumset: Jia's Lemma}
\label{sec:differences}
We now turn to the critical task: controlling the differences of a $k$-fold sumset to produce the upper bound on $\sum_{\bar y} \binom{G_N(\bar y)}{2}$ required by Lemma~\ref{lem:energy}.

\begin{rem}{Why truncate?}{truncate}
We need to truncate before building the $k$-fold sumset when working in $\ZZ^d$. The purpose of this remark is to clarify this need.

\indent
In $\NN^d$ with $K=[0,1]^d$, if $\Sigma\beta\in[0,m)^d$ then every element of $\beta$ lies in $[0,m)^d$, and so the full sumset $k\cA$ has $(k\cA)\cap[0,m)^d=\cS^{(m)}$, whose size is $O(m^{d/2})$ by~\eqref{eq:S upper}. Moreover, every summand of an element of $k\cA$ of norm less than $m$ is itself of norm less than $m$. Nothing of the kind holds in $\ZZ^d$. 

\indent 
Let $\cJ\subseteq\NN$ be an infinite $B_4$-set, let $(N_j)_{j\in\cJ}$ be positive integers with $N_j>8\sum_{j'\in\cJ, j'<j}N_{j'}$, and let
  \[\cA\coloneqq\{(j,N_j) : j\in\cJ\}\cup\{(0,-N_j) : j\in\cJ\}\subseteq\ZZ^2.\]
A $4$-multiset from $\cA$ with $p_j$ copies of $(j,N_j)$ and $q_j$ copies of $(0,-N_j)$ has sum $\bigl(\sum_jp_jj,\sum_j(p_j-q_j)N_j\bigr)$. If two such multisets have the same sum, then the lacunarity of $(N_j)$ forces $p_j-q_j=p'_j-q'_j$ for every $j$ (consider the largest $j$ at which they differ); then $\sum_jp_j=\sum_jp'_j$ because both multisets have four elements, and $\sum_jp_jj=\sum_jp'_jj$ forces $p=p'$ because $\cJ$ is a $B_4$-set and hence (Lemma~\ref{lem:B_h is B_p}) a $B_s$-set for $s\le4$. So $\cA$ is a $B_4$-set. But $2\cA\ni(j,N_j)+(0,-N_j)=(j,0)$ for every $j\in\cJ$, so $2\cA$ has at least $\abs{\cJ\cap[0,m)}$ elements in $(-m,m)^2$, which is unbounded, while $\cA$ has at most $2\#\{j:N_j<m\}$ elements in $(-m,m)^2$, which grows as slowly as we like.
\end{rem}

The remedy is to never use the full sumset $k\cA$. For $N\ge 2$ put
  \begin{equation}\label{eq:SN def}
    W\coloneqq\ceiling{N^2/\log^3 N},\qquad \cS_N\coloneqq k\cA_W=\bigl\{\Sigma\beta:\beta\in\multibinom{\cA_W}{k}\bigr\},
  \end{equation}
a finite set, one for each $N$; this is why Lemma~\ref{lem:energy} is stated for a family. For $\bar s\in\cS_N$ we write $\lambda_{\bar s}\in\multibinom{\cA_W}{k}$ for its unique representation (Corollary~\ref{cor:Vs} applied to the $B_h$-set $\cA_W$). Two features of $\cS_N$ are immediate: $\cS^{(m)}\subseteq\cS_N\cap\{\bar x \in \RR^d : \rho(\bar x)<m\}$ whenever $m\le W$, since then $\cA_m\subseteq\cA_W$; and every element of every $\lambda_{\bar s}$ has $\rho<W$, by definition. For $K=[0,1]^d$ and $\cA\subseteq\NN^d$, $\cS_N$ agrees with $k\cA$ on $[0,W)^d$.

We now begin our statement and proof of a quantitative form of Jia's Lemma. We will not need the form of the constants $c_0$, $C(r,p,q)$, but we find including them makes the argument easier to follow.

\begin{dfn}{$T$ and $\Phi$}{TandPhi}
For nonnegative integers $r,p,q$ and real $L\ge1$, let
\begin{align*}
  T(p,q;L)
  &\coloneqq \{(\beta,\delta)\in \multibinom{\cA}{p} \times \multibinom{\cA}{q} :
     \beta\cap\delta=\emptyset,\ \rho(\Sigma\beta-\Sigma\delta)<L\} \\
  \Phi(r,p,q;L)
   &\coloneqq \left\{(\vec a;\beta,\delta):
     \vec a \in (\cA_L)^r ,
     (\beta,\delta) \in T(p,q;L)\right\}.
\end{align*}
We note that $T$ is defined using all of $\cA$, not a truncation. 
\end{dfn}
Note that $\beta,\delta$ are multisets taken from $\cA$, not from $\cA_L$, but $\vec a = (\bar a_1,\dots,\bar a_r )$ is an ordered tuple of elements from $\cA_L$, possibly with repetitions. This asymmetry is useful because the
$\bar a_i$ range freely over $\cA_L$ and that allows us to factor nicely:
\begin{equation}\label{eq:factor Phi}
   \abs*{\Phi(r ,p,q;L)} = A_K(L)^r \abs*{T(p,q;L)}.
\end{equation}
Here $T(p,q;L)$ is finite whenever $p+q\le h$; this is the $r=0$ case of Lemma~\ref{lem:jia}, whose proof does not use~\eqref{eq:factor Phi}.

An important insight of Jia~\cite{1994.Jia} is that one needs to bound $\abs{\Phi}$, not $|T|$ alone.
Jia gives the following lemma (for $d=1$) with the conclusion $\abs{\Phi} = O(L)$.
We have made it quantitative to ease our concerns over which constants depend on which, and the circular reasoning such confusion can enable.
Certainly, the specific form of $C(r,p,q)$, which depends also on $d$ and $K$, is not germane to our usage.

\begin{lem}{Jia's Lemma}{jia}
Let $\cA\subseteq\ZZ^d$ be a $B_{h}$-set (with $h$ not necessarily even) and let $r,p,q\ge0$ satisfy $r+p+q\le h$.
Then for every real $L \ge 1$,
\[
   \abs*{\Phi(r,p,q;L)} \le L^d \cdot C(r,p,q),
\]
where \(c_0 \coloneqq V(r+1+c_K)^d\frac{(r+p)!}{p!}\) and \(C(r,p,q) \coloneqq c_0 \sum_{j=0}^{q}(2^dr)^{j}.\)
\end{lem}

\begin{proof}
We induct on $q$, using $q=0$ as our base case.

With $q=0$, any $(\beta,\delta)\in T(p,0;L)$ has $\delta=\emptyset$ and $\rho(\Sigma\beta)<L$. Consider the map taking
$(\vec a;\beta)$ to $\bar m\coloneqq \Sigma \vec a + \Sigma\beta\in\ZZ^d$.  By subadditivity, $\rho(\bar m)<rL+L$, and by Lemma~\ref{lem:lattice} there are at most $V((r+1)L+c_K)^d\le V(r+1+c_K)^dL^d$ possibilities for $\bar m$.  If two tuples share
$\bar m$, then $\multi{\bar a_1,\dots,\bar a_r } \uplus \beta$ and $\multi{\bar a_1',\dots,\bar a_r'}\uplus\beta'$ are size $(r+p)$ multisets with the same sum $\bar m$. As $r+p\le h$ and $\cA$ is a $B_h$-set,
  \[\multi{\bar a_1,\dots,\bar a_r}\uplus\beta = \multi{\bar a_1',\dots,\bar a_r'}\uplus\beta';\]
call this multiset $G$.
The elements of the multiset $G$ can be listed in at most $(r+p)!$ ways, and the ordering of the entries that will be put into $\beta$ is irrelevant, so there are at most $(r+p)!/p!$ tuples $(\vec a;\beta) \in \cA^r \times \multibinom{\cA}{p}$ with a given $\bar m$.  Hence
  \[\abs*{\Phi(r ,p,0;L)} \le \frac{(r +p)!}{p!} V(r+1+c_K)^dL^d=c_0L^d = L^d \cdot C(r,p,0),\]
which establishes the base of our induction.

We now assume that $q\ge 1$, that for every $L\ge 1$ we have
  \[\abs{\Phi(r,p,q-1;L)} \le L^d\cdot C(r,p,q-1),\]
and that $r+p+q \le h$, and will show that
  \[\abs*{\Phi(r,p,q;L)} \le L^d\cdot C(r,p,q) = c_0L^d + r \cdot (2L)^d\cdot c_0 \sum_{j=0}^{q-1}(2^dr )^j.\]

We split the count into two cases: $((\bar a_1,\dots,\bar a_r );\beta,\delta)\in \Phi(r,p,q;L)$ where no $\bar a_i$ is in $\delta$, and where some $\bar a_i \in\delta$ and we can use the induction hypothesis.

First, we suppose that no $\bar a_i$ is in $\delta$. Map
$(\vec a;\beta,\delta)\in \Phi(r,p,q;L)$ to $\Sigma\vec a+\Sigma\beta - \Sigma\delta =: \bar m\in\ZZ^d$. By subadditivity $\rho(\bar m)<rL+L$, so again there are at most $c_0L^dp!/(r+p)!$ choices for $\bar m$.  If two tuples map to the same $\bar m$, then we have two multisets with $r+p+q$ elements in a $B_{r+p+q}$-set (since $r+p+q \le h$) having the same sum, and so the multisets are equal:
  \[\multi{\bar a_1,\dots,\bar a_r } \uplus \beta \uplus \delta'=\multi{\bar a_1',\dots,\bar a_r '} \uplus \beta'\uplus\delta .\]
Put $P=\multi{\bar a_1,\dots,\bar a_r } \uplus \beta$, $Q=\delta$, $P'=\multi{\bar a_1',\dots,\bar a_r '}\uplus\beta'$, $Q'=\delta'$ into Lemma~\ref{lem:cancel}, and we conclude that $P=P'$ and $\delta=\delta'$. Thus each value of $\bar m$ determines $\delta$ and the unique multiset $\multi{\bar a_1,\dots,\bar a_r } \uplus \beta$, and the latter splits into $(\vec a;\beta)$ in at most $(r +p)!/p!$ ways. Thus, there are at most $c_0L^d$ tuples $(\vec a;\beta,\delta)$ with no $\bar a_i$ in $\delta$.

The case where some $\bar a_i\in\delta$ is a touch more involved. To start, we can make a count with the assumption that $\bar a_1\in \delta$, and then multiply by $r$ to get an upper bound on the number of tuples with \emph{some} $\bar a_i\in\delta$. Delete one copy of $\bar a_1$ from $\delta$ to get $\delta^- \in \multibinom{\cA}{q-1}$ and $\multi{\bar a_1}\uplus \delta^- = \delta$. We have $\beta \cap \delta^-=\emptyset$ and
  \[\rho(\Sigma\beta-\Sigma\delta^-) = \rho\bigl((\Sigma\beta-\Sigma\delta)+\bar a_1\bigr)<L+L=2L,\]
so $(\vec a; \beta, \delta^-) \in \Phi(r ,p,q-1;2L)$ (note $\cA_L\subseteq\cA_{2L}$).
By the induction hypothesis,
  \[\abs*{\Phi(r ,p,q-1;2L)} \le (2L)^d \cdot C(r,p,q-1)
  = (2L)^d \cdot c_0\sum_{j=0}^{q-1} (2^dr )^j.\]
Putting the first element of $\vec a$ back into $\delta$ does not increase the count, which stands at: at most
  \[ r  \cdot (2L)^d \cdot c_0 \sum_{j=0}^{q-1} (2^dr )^j.\]

The two cases combine to give at most
  \[\abs*{\Phi(r ,p,q;L)} \le c_0L^d + r \cdot (2L)^d \cdot c_0\sum_{j=0}^{q-1} (2^dr )^j
    = L^d \cdot c_0 \sum_{j=0}^{q} (2^dr )^j,\]
as needed to complete the induction.
\end{proof}

\subsection{Stratifying the pairs of elements from \texorpdfstring{$\cS_N$}{S sub N}}
We stratify the pairs of elements of $\cS_N$ by how many summands their representations share. For $L>0$ and $0\le r\le k$ let
  \[P_r(L)\coloneqq\#\bigl\{(\bar s,\bar s')\in\cS_N^2 : \bar s\ne\bar s',\ \abs{\lambda_{\bar s}\Cap \lambda_{\bar s'}}=r,\ \rho(\bar s-\bar s')<L\bigr\}.\]
As $\abs{\lambda_{\bar s}\Cap \lambda_{\bar s'}}=k$ forces $\lambda_{\bar s}=\lambda_{\bar s'}$ and $\bar s=\bar s'$, we have $P_k(L)=0$.

\begin{lem}{Strata}{strata}
Let $h=2k$ be even, let $\cA\subseteq\ZZ^d$ be a $B_h$-set satisfying the growth hypothesis~\eqref{dfn:growth} for $\tau$, and let $W$ and $\cS_N$ be as in~\eqref{eq:SN def}.
\begin{enumerate}[label=\textup{(\alph*)},noitemsep]
  \item\label{strata:rigid} For each $\bar v\in\ZZ^d\setminus\{0\}$ there is at most one pair $(\bar s,\bar s')\in\cS_N^2$ with $\bar s-\bar s'=\bar v$ and $\lambda_{\bar s}\Cap \lambda_{\bar s'}=\emptyset$. Consequently $P_0(L)\le V(L+c_K)^d$ for every $L>0$.
  \item\label{strata:small} There is a function $\eta(N)$, depending only on $d,h,K,\tau,n_0$, with $\eta(N)\to0$ as $N\to\infty$, such that for all $N\ge n_0$, all $1\le r\le k-1$, and all real $L\ge N$,
    \[P_r(L)\le\eta(N)L^d.\]
\end{enumerate}
\end{lem}

\begin{proof}
\ref{strata:rigid} Suppose that $(\bar s,\bar s')$ and $(\bar u,\bar u')$ are two such pairs. Then $\bar s+\bar u'=\bar u+\bar s'$, an identity of $2k$-fold sums of elements of the $B_{2k}$-set $\cA_W$, so
  \[\lambda_{\bar s} \uplus \lambda_{\bar u'} = \lambda_{\bar u} \uplus \lambda_{\bar s'}.\]
By Lemma~\ref{lem:cancel}, $\lambda_{\bar s}=\lambda_{\bar u}$ and $\lambda_{\bar s'}=\lambda_{\bar u'}$, from which it follows that $\bar s=\bar u$ and $\bar s'=\bar u'$. Thus $P_0(L)$ is at most the number of nonzero $\bar v\in\ZZ^d$ with $\rho(\bar v)<L$, which is at most $V(L+c_K)^d$ by Lemma~\ref{lem:lattice}.

\ref{strata:small} Consider a pair $(\bar s,\bar s')$ counted by $P_r(L)$, so that there are unique $\alpha\in \multibinom{\cA_W}{r}$ and $\beta,\delta \in \multibinom{\cA}{k-r}$ with $\alpha =\lambda_{\bar s} \Cap \lambda_{\bar s'}$, $\lambda_{\bar s} = \alpha \uplus \beta$ and $\lambda_{\bar s'} = \alpha \uplus \delta$.
Note that $\beta \cap \delta = \emptyset$, since at each $\bar x$ either $\mult_\beta(\bar x)=0$ or $\mult_\delta(\bar x)=0$.
The map from $(\bar s,\bar s')$ to $(\alpha;\beta,\delta)$ is injective, and since $\Sigma \beta -\Sigma\delta =\bar s-\bar s'$, we have $(\beta,\delta) \in T(k-r,k-r;L)$. Hence, using $r\ge 1$ and the factorization~\eqref{eq:factor Phi},
\begin{align*}
   P_r(L)
   &\le \abs*{\multibinom{\cA_W}{r}}\cdot\abs*{T(k-r,k-r;L)} \\
   &\le A_K(W)^{r}\cdot\frac{\abs*{\Phi(2r,k-r,k-r;L)}}{A_K(L)^{2r}} \\
   &\le C(2r,k-r,k-r)L^d\left(\frac{A_K(W)}{A_K(L)^2}\right)^r, 
\end{align*}
by Lemma~\ref{lem:jia}, which applies since $2r+2(k-r)=h$. As $L\ge N$ and $A_K$ is nondecreasing, we may replace $A_K(L)$ by $A_K(N)$.

It remains to see that the \emph{regularity ratio} $A_K(W)/A_K(N)^2$ tends to $0$. This is why Jia introduced the hypothesis $A(N^2) = O(A(N)^2)$. As Erd\H{o}s, Helm, and Jia weren't thinking of producing explicit constants, they never separated the number of blocks from their width, and just used $W=N^2$. The flexibility afforded by this additional parameter is what allows our argument to proceed without Jia's regularity hypothesis and without Chen's case analysis. That is, this step is why we have $W=N^2/\log^3 N$ and not $W=N^2/\log N$ (as in Part I of this series of papers) or $W=N^2$ (as in Chen's work). In dimension $d\ge2$ there is more room ($W=N^2/\log^{2}N$ would also work), but we see no reason to use it.

By Lemma~\ref{lem:easy Bh bound} and $W\le 2N^2/\log^3 N$,
\[
   A_K(W)\le(h!V)^{1/h}(hW+c_K)^{d/h}
   =O\left(\frac{N^{2d/h}}{\log^{3d/h} N}\right),
\]
while~\eqref{dfn:growth} gives $A_K(N)\ge \tau V^{1/h}N^{d/h}/\log^{1/h} N$ for
$N\ge n_0$.  Therefore
\[
   \frac{A_K(W)}{A_K(N)^{2}}
   =O\left(\frac{\log^{-3d/h} N}{\log^{-2/h} N}\right)
   =O\left(\log^{-(3d-2)/h} N\right),
\]
which goes to $0$ as $N\to \infty$, since $3d-2\ge1$. Taking 
  \[\eta(N) \coloneqq \max_{1\le r\le k-1}C(2r,k-r,k-r)
  \left(\frac{A_K(W)}{A_K(N)^2}\right)^r\] 
completes the proof.
\end{proof}

The Strata Lemma has two consequences. The first is the energy bound required by Lemma~\ref{lem:energy}.

\begin{lem}{Energy of the sumset}{sumset energy bound}
Let $h=2k$ be even, let $\cA\subseteq\ZZ^d$ be a $B_h$-set satisfying the growth hypothesis~\eqref{dfn:growth}, let $W$ and $\cS_N$ be as in~\eqref{eq:SN def}, and let $G_N(\bar y)=\abs{\cS_N\cap(\bar y+[0,N)^d)}$ be the cube counts of $\cS_N$. Then, as $N\to\infty$,
  \[   \sum_{\bar y\in\ZZ^d}\binom{G_N(\bar y)}{2} \le \frac{1}{2}N^{2d} + o(N^{2d}).\]
\end{lem}

\begin{proof}
The sum counts the pairs $(\bar y,\{\bar s,\bar s'\})$ with $\bar s\ne \bar s'$ both in $\cS_N\cap(\bar y+[0,N)^d)$. For a fixed pair $(\bar s,\bar s')$ with $\bar s-\bar s'=\bar v\in\ZZ^d$, the number of $\bar y\in\ZZ^d$ with $\bar s,\bar s'\in \bar y+[0,N)^d$ is $\prod_{i=1}^d(N-\abs{v_i})_+$, where $(x)_+\coloneqq\max\{x,0\}$. Hence
\begin{equation}\label{eq:average}
   \sum_{\bar y\in\ZZ^d}\binom{G_N(\bar y)}{2}
    = \frac12\sum_{\substack{\bar v\in\ZZ^d\setminus\{0\}\\ \norm{\bar v}<N}}\prod_{i=1}^d(N-\abs{v_i})\cdot P(\bar v),
\end{equation}
where $P(\bar v)\coloneqq\#\{ (\bar s,\bar s') \in \cS_N^2 : \bar s-\bar s'=\bar v \}$ is the number of ordered pairs with difference $\bar v$; the factor $\frac12$ accounts for each unordered pair $\{\bar s,\bar s'\}$ being counted at both $\bar v$ and $-\bar v$.

Write $P(\bar v)=\sum_{r=0}^{k-1}P_r(\bar v)$, where $P_r(\bar v)$ counts the pairs with $\abs{\lambda_{\bar s}\Cap \lambda_{\bar s'}}=r$. By Lemma~\ref{lem:strata}\ref{strata:rigid}, $P_0(\bar v)\le1$ for $\bar v\ne0$, and since $\sum_{j=-(N-1)}^{N-1}(N-\abs j)=N^2$, we have $\sum_{\norm{\bar v}<N}\prod_i(N-\abs{v_i})=N^{2d}$, so the $r=0$ part of~\eqref{eq:average} is at most $\frac12N^{2d}$. For $1\le r\le k-1$ we use $\prod_i(N-\abs{v_i})\le N^d$, and note that $\norm{\bar v}<N$ implies $\rho(\bar v)<c_KN$ by~\eqref{eq:gauge vs norm}, so that $\sum_{\norm{\bar v}<N}P_r(\bar v)\le P_r(c_KN)\le\eta(N)c_K^dN^d$ by Lemma~\ref{lem:strata}\ref{strata:small}, recalling $c_K\ge1$. Altogether, as $\eta(N)=o(1)$ by Lemma~\ref{lem:strata}\ref{strata:small},
  \[\sum_{\bar y\in\ZZ^d}\binom{G_N(\bar y)}{2}\le\frac12N^{2d}+\frac{N^d}{2}\sum_{r=1}^{k-1}\eta(N)c_K^dN^d=\frac12N^{2d}+o(N^{2d}).\qedhere\]
\end{proof}

The second consequence replaces the bound $\abs{(k\cA)\cap[0,m)^d}=O(m^{d/2})$, which is available in $\NN^d$ but not in $\ZZ^d$, and supplies hypotheses~\ref{eng:window} and~\ref{eng:head} of Lemma~\ref{lem:energy}.

\begin{lem}{Uniform upper bound}{uniform}
Let $h=2k$ be even, let $\cA\subseteq\ZZ^d$ be a $B_h$-set satisfying the growth hypothesis~\eqref{dfn:growth}, and let $W$ and $\cS_N$ be as in~\eqref{eq:SN def}. Then $\abs{\cS_N}=O(N^d/\log^{3d/2}N)$, and
  \[\#\bigl(\cS_N\cap\{\bar x \in \RR^d : \rho(\bar x)<m\}\bigr)=O(m^{d/2})\]
uniformly for all $N\ge n_0$ and all $m\ge N$.
\end{lem}

\begin{proof}
By Corollary~\ref{cor:Vs} and Lemma~\ref{lem:easy Bh bound}, $\abs{\cS_N}\le(A_K(W)+k)^k/k!=O(W^{kd/h})=O(W^{d/2})$, which is the first claim. For $m\ge W$ the second claim follows, as $\cS_N\cap\{\rho<m\}\subseteq\cS_N$ has $O(W^{d/2})=O(m^{d/2})$ elements. Let $N\le m\le W$ and $t\coloneqq\#(\cS_N\cap\{\rho<m\})$. The $t(t-1)$ ordered pairs of distinct elements $\bar s,\bar s'$ of $\cS_N\cap\{\rho<m\}$ have $\rho(\bar s-\bar s')\le\rho(\bar s)+\rho(-\bar s')\le\rho(\bar s)+\sigma_K\rho(\bar s')<(1+\sigma_K)m$, so with $L\coloneqq(1+\sigma_K)m\ge N$, Lemma~\ref{lem:strata} gives
  \[t(t-1)\le\sum_{r=0}^{k-1}P_r(L)\le V(L+c_K)^d+(k-1)\eta(N)L^d=O(m^d),\]
uniformly in $N\ge n_0$ and $m\ge N$. Thus $t=O(m^{d/2})$.
\end{proof}

\section{The proof of Theorem~\ref{thm:main}}\label{sec:main proof}
We now assemble the lemmata of the previous section into a proof of Theorem~\ref{thm:main}, first when $\bar 0$ is interior to $K$, and then in general.

\begin{proof}
We first assume that $\bar 0$ is in the interior of $K$.

Let $k=h/2$ and $V=\vol(K)$. Assume, by way of contradiction, that
  \begin{equation}\label{eq:contradiction hypothesis}
    A_K(x) \ge \tau \left(\frac{Vx^d}{\log x}\right)^{1/h}\ \text{ for all } x\ge n_0\ge 3,
    \qquad \text{with} \qquad
    \tau > C_{d,h}.
  \end{equation}
This is the growth hypothesis~\eqref{dfn:growth}, so every lemma of Section~\ref{sec:lemmata} applies. Let $W=\ceiling{N^2/\log^3N}$ and $\cS_N=k\cA_W$ be as in~\eqref{eq:SN def}, and let $\cB_N\coloneqq\cS_N$, with $B_N(m)=\#(\cS_N\cap\{\bar x \in \RR^d : \rho(\bar x) <m\})$. We check the hypotheses of Lemma~\ref{lem:energy}. Certainly $\log W/\log N\to2$. By Lemma~\ref{lem:uniform},
\begin{enumerate}[label=\textit{(\roman*)},noitemsep]
  \item $\abs{\cB_N}= O(N^d/\log^{3d/2} N) = o(N^d)$;
  \item $B_N((2+c_K)N) = O(N^{d/2}) = o(\sqrt{N^d\log N})$;
\end{enumerate}
and by Lemma~\ref{lem:sumset energy bound},
\begin{enumerate}[label=\textit{(\roman*)},noitemsep,start=3]
  \item $\displaystyle\sum_{\bar y\in\ZZ^d} \binom{G_N(\bar y)}{2} \le \frac12 N^{2d} + o(N^{2d})$,
\end{enumerate}
that is, \ref{eng:energy} holds with $c=\frac12$.

Set
  \[\tau'\coloneqq\frac{\tau^k \Gamma(1+\tfrac dh)^k}{k!\Gamma(1+\tfrac d2)}\]
and fix $\epsilon\in(0,1)$. By Lemma~\ref{lem:simplex} there is $m_0$ such that $\abs{\cS^{(m)}}\ge(1-\epsilon)\tau'\sqrt{Vm^d/\log m}$ for all $m\ge m_0$. For $m\le W$ we have $\cS^{(m)}\subseteq\cS_N\cap\{\bar x \in \RR^d : \rho(\bar x) <m\}$, and so for all $N$ with $N\log N\ge m_0$,
  \[B_N(m)\ge(1-\epsilon)\tau'\sqrt{\frac{Vm^d}{\log m}}\qquad\text{for all } m \text{ with } N\log N\le m\le W,\]
which is~\eqref{hyp:profile} with $(1-\epsilon)\tau'$ in place of $\tau$. Lemma~\ref{lem:energy} with $c=\frac12$ yields $(1-\epsilon)\tau'\le\sqrt{4/(d\log 2)}$, and as $\epsilon$ is arbitrary,
  \[\tau^k \le \sqrt{\frac{4}{d\log 2}} \cdot\frac{k!\Gamma(1+\tfrac d2)}{\Gamma(1+\tfrac dh)^k}.\]
Squaring both sides gives
  \[\tau^h \le \frac{4}{d\log 2} \cdot \frac{(k!)^2 \Gamma(1+\tfrac d2)^2}{\Gamma(1+\tfrac dh)^h},\]
and so $\tau\le C_{d,h}$ by Definition~\ref{dfn:C}, as $k!=\Gamma(1+\tfrac h2)$. This contradicts~\eqref{eq:contradiction hypothesis}.

Now let $K$ be a convex body with $\bar 0\in K$, possibly on the boundary and not interior. For $\epsilon>0$ let $K_\epsilon\coloneqq K+[-\epsilon,\epsilon]^d$, a convex body with $\bar 0$ in its interior. As $K$ is closed, $\bigcap_{\epsilon>0}K_\epsilon=K$, and so $\vol(K_\epsilon)\to\vol(K)$ as $\epsilon\to0$ by continuity of Lebesgue measure. For every $n$, $\bigcup_{t<n}t\ast K\subseteq\bigcup_{t<n}t\ast K_\epsilon$, whence $A_K(n)\le A_{K_\epsilon}(n)$, and therefore
  \begin{align*}
  \liminf_{n\to\infty}\frac{A_K(n)}{(\vol(n\ast K)/\log n)^{1/h}}
  &\le\left(\frac{\vol(K_\epsilon)}{\vol(K)}\right)^{1/h}\liminf_{n\to\infty}\frac{A_{K_\epsilon}(n)}{\left( \vol(n\ast K_\epsilon)/\log n \right)^{1/h}} \\
  &\le\left(\frac{\vol(K_\epsilon)}{\vol(K)}\right)^{1/h}C_{d,h}
  \end{align*}
by the case already proved. Letting $\epsilon\to0$ completes the proof.
\end{proof}

\begin{rem}{}{log density}
The proof gives slightly more than the statement of Theorem~\ref{thm:main}. The contradiction hypothesis~\eqref{eq:contradiction hypothesis} is used only for $x$ in a window: in Lemma~\ref{lem:strata} at $x=N$ and $x=W$, in Lemma~\ref{lem:energy} for $N\log N\le m\le W$, and in Lemma~\ref{lem:simplex} for $n_0\le x\le m\le W$. Moreover, Lemma~\ref{lem:simplex} appears to tolerate $n_0=n_0(N)$ with $n_0=o(N)$: the cost of truncating the Dirichlet integral at $u_0\approx n_0$ is a factor $(u_0/m)^{d/h}$ with $m\ge N\log N$, which is $o(1)$ provided $n_0=o(N)$. Tracking the dependence of the $o(1)$ terms on $N$ might then show that for every $\epsilon>0$ and all sufficiently large $N$, there is some
  \[x\in\left[\frac{N}{\log N},\ N^2\right]\qquad\text{with}\qquad A_K(x)\le(C_{d,h}+\epsilon)\left(\frac{\vol(x\ast K)}{\log x}\right)^{1/h}.\]
That is, the set of scales at which $\cA$ is thin would have positive logarithmic density, rather than merely being infinite, and Corollaries~\ref{cor:gaps} and~\ref{cor:holes} would hold with ``for infinitely many $n$'' sharpened to ``for some $n'\in[n,n^{2+o(1)}]$, for every large $n$.'' Erd\H os's original argument for Sidon sets has the same feature. We have not checked the details, and mention this only as a possible further direction.
\end{rem}

\section{Further problems}\label{sec:problems}

\subsection{The constant}
Even for $d=1$ and $h=2$ we do not expect $C_{1,2}=2/\sqrt{\log 2}$ to be the truth. Of the three factors in $C_{d,h}^h$ displayed after Theorem~\ref{thm:main}, the simplex and dimension factors arise from computations that are sharp for the hypothetical extremal profile, but the energy factor $4/(d\log 2)$ comes from Cauchy's inequality, which is an equality only if the cube counts $G_N$ of the sumset decay exactly along the weight $w$, and from the energy bound of Lemma~\ref{lem:sumset energy bound}, which is an equality only if every lag $\bar v$ with $\norm{\bar v}<N$ occurs as a difference in the window. Nothing forces either, and we suspect that $\left(n^d/\log n\right)^{1/h}$ is not even the right normalization. Jia~\cite{1994.Jia} conjectured that $\liminf_n A(n)(\log n/n)^{1/h}<\infty$ for every $B_h$-set $\cA\subseteq\NN$, even or odd $h$. We conjecture more.
\begin{cnj}{Infinite $B_h$-sets lose more than a logarithm}{log not enough}
  Let $d,h$ be positive integers with $h\ge 2$. If $\cA\subseteq\NN^d$ is a $B_h$-set, then
  \[ \liminf_{n\to\infty} \frac{A(n)}{(n^d/\log n)^{1/h}} = 0.\]
\end{cnj}
For $h$ even, a proof of Conjecture~\ref{conj:log not enough} by the methods of this paper would need some property of $k\cA$ beyond the near-distinctness of its small differences, as that is the only property of $k\cA$ used here; or else an improvement in the Sidon case itself, since the argument is the $h=2$ case applied to the $k$-fold sumset.

A different kind of improvement concerns the set of scales. Theorem~\ref{thm:main} produces infinitely many $n$ at which $\cA$ is thin, and Remark~\ref{rem:log density} suggests that the proof can be pushed to give one such $n$ in every interval $[N/\log N, N^2]$. We have not carried this out, and we do not know the true density of the thin scales: must they have positive upper density, for example? The same question applies to the dense scales in Corollary~\ref{cor:limsup}, and to the constructions discussed in Section~\ref{sec:limsup problems} below.

\subsection{Odd \texorpdfstring{$h$}{h}}
For odd $h$, no result like Theorem~\ref{thm:main} is known in any dimension; Conjecture~\ref{conj:log not enough} is open, and so is Jia's conjecture, and so is the weaker statement that
  \begin{equation}\label{conj:odd h}
    \liminf_{n\to\infty} \frac{A(n)}{n^{d/h}} = 0
  \end{equation}
for every $B_h$-set $\cA\subseteq\NN^d$ with $h\ge3$ odd. Every $B_h$-set is a $B_{h-1}$-set, but applying the even case to a $B_h$-set with $h$ odd says nothing, since then $A(n)=O(n^{d/h})=o\bigl((n^d/\log n)^{1/(h-1)}\bigr)$ trivially. The obstacle is that every proof for even $h$ rests on the differences of the $(h/2)$-fold sumset, and for odd $h$ there is no such sumset. In Part III of this series~\cite{2026.Obryant-c}, the author will bound $\liminf_n A(n)/n^{1/h}$ in terms of $\limsup_n A(n)/n^{1/h}$ for odd $h$ and $d=1$, and will examine how far the known techniques reach; they do not reach~\eqref{conj:odd h}.

\subsection{Dense scales and constructions}\label{sec:limsup problems}
Theorem~\ref{thm:main} is a statement about the $\liminf$. In the other direction there are two questions: how thick can a $B_h$-set be at every scale (the $\liminf$ exponent), and how thick can it be at infinitely many scales (the $\limsup$)?

For the $\liminf$, the best construction for $d=1$ is Cilleruelo's~\cite{2014.Cilleruelo}: an explicit $B_h$-set $\cG\subseteq\NN$ with
  \[G(n) = n^{\sqrt{(h-1)^2+1}-(h-1)+o(1)},\]
which for $h=2$ recovers the exponent $\sqrt2-1$ of Ruzsa~\cite{1998.Ruzsa}. For $d\ge2$ and $h=2$, Cilleruelo's transfer theorem~\cite[Theorem 7]{2010.Cilleruelo-a}, mentioned in Section~\ref{sec:history}, gives Sidon sets $\cG\subseteq\NN^d$ with $G(n)=n^{d(\sqrt2-1)+o(1)}$, and shows that this is the best exponent in dimension $d$ if and only if $\sqrt2-1$ is the best exponent in dimension $1$. His maps between $\NN$ and $\NN^d$ interleave digits in mixed-radix systems, and preserve the Sidon property because the digits are small enough that a sum of two elements produces no carries. A sum of $h$ elements produces no carries once the radices are enlarged by a factor of $h$, which costs only a constant factor in each coordinate, so we expect the following to have a positive answer.
\begin{qst}{The exponent-dimension relationship}{ed}
  Let $\alpha_d^{(h)}$ be the supremum of $\liminf_{n\to\infty}\log G(n)/\log n$ over all infinite $B_h$-sets $\cG\subseteq\NN^d$. Is $\alpha_d^{(h)}=d\alpha_1^{(h)}$ for all $h\ge 3$?
\end{qst}
A positive answer, with Cilleruelo's construction, would give $B_h$-sets $\cG\subseteq\NN^d$ with $G(n)\ge n^{d(\sqrt{(h-1)^2+1}-(h-1))+o(1)}$, to be compared with the exponent $d/h$ in Corollary~\ref{cor:cube}.

For the $\limsup$, Kr\"uckeberg~\cite{1961.Kruckeberg} constructed Sidon sets $\cG\subseteq\NN$ with $\limsup_n G(n)/\sqrt n\ge1/\sqrt2$, and Part I~\cite{2026.Obryant-a} extends this to $g$-Golomb rulers. Wang \& Yang~\cite{2026.Wang&Yang} have recently shown that for every $h\ge2$, and every finite $B_h$-set $F\subseteq\NN$, there is an infinite $B_h$-set $\cG\supseteq F$ with
  \[\limsup_{n\to\infty} \frac{G(n)}{n^{1/h}} = \frac 1h
  \qquad\text{and}\qquad
  \liminf_{n\to\infty} \frac{G(n)}{n^{1/h}} = 0.\]
So the dense scales of Corollary~\ref{cor:limsup} do occur, for every $h$ in dimension one, and the constant $1/h$ is not claimed to be optimal. The trivial bound of Lemma~\ref{lem:easy Bh bound} caps the $\limsup$ at $(h\cdot h!)^{1/h}$, and the finite results of Johnston \& Tait \& Timmons~\cite{2022.Johnston&Tait&Timmons} cap it lower; the gap between $1/h$ and those caps is wide open. We are not aware of any $B_h$-set $\cG\subseteq\NN^d$ with $d\ge2$ and $h\ge3$ having $\limsup_n G(n)/n^{d/h}>0$. The digit-interleaving maps of Question~\ref{qst:ed} do not obviously help, because they distort the counting function by $n^{o(1)}$ factors, which is harmless for exponents but fatal for constants.
\begin{qst}{Dense scales in higher dimensions}{limsup d}
  For $d\ge2$ and $h\ge3$, is there a $B_h$-set $\cG\subseteq\NN^d$ with $\limsup_{n\to\infty} G(n)/n^{d/h}>0$? If so, what is the supremum of this $\limsup$, as a function of $d$ and $h$?
\end{qst}

\subsection{\texorpdfstring{$B_h[g]$}{B sub h of g}-sets}
A set $\cA\subseteq\NN^d$ is a $B_h[g]$-set if every point of $\NN^d$ is $\Sigma\beta$ for at most $g$ multisets $\beta\in\multibinom{\cA}{h}$; a $B_h[1]$-set is a $B_h$-set, and a $B_h$-set is a $B_h[g]$-set for every $g$. The argument of Lemma~\ref{lem:easy Bh bound} gives $A(n)=O(n^{d/h})$ for every $B_h[g]$-set, and for finite sets this is the right order: Johnston \& Tait \& Timmons~\cite{2022.Johnston&Tait&Timmons} (and implicitly Caicedo \& G\'omez \& Trujillo~\cite{2015.Caicedo&Gomez&Trujillo}) construct $B_h[g]$-sets in $[1,n]$ with $(1-o(1))(gn)^{1/h}$ elements, while the best upper bound for $g\ge2$ and large $h$ is $\bigl(\sqrt{\pi h/2}\cdot h!g(1+o(1))n\bigr)^{1/h}$~\cite{2022.Johnston&Tait&Timmons}.

For infinite $B_h[g]$-sets with $g\ge2$, nothing beyond the trivial bound is known, in any dimension and for any $h$. The case $d=1$, $h=2$ is a problem of Erd\H{o}s \& S\'ark\"ozy \& S\'os~\cite[Problem 9]{1994.Erdos-Sarkozy-Sos}, and is Problem 158 in Bloom's database of Erd\H os problems~\cite{ErdosProblems158}, where the discussion leans toward the answer we conjecture. The methods of this paper do not apply: the rigidity of Lemma~\ref{lem:strata}\ref{strata:rigid} fails for $g\ge2$, since a $B_h[g]$-set can have each difference of $(h/2)$-fold sums repeated up to $g$ times. Part I handles $g$-Golomb rulers, a different generalization, precisely because bounded repetition of \emph{differences} is what the argument tolerates. Based on greedy computations, we nonetheless conjecture the following.
\begin{cnj}{Infinite $B_h[g]$-sets cannot stay dense}{Bhg}
  Let $d,h,g$ be positive integers with $h\ge 2$. If $\cA\subseteq\NN^d$ is a $B_h[g]$-set, then
    \[\liminf_{n\to\infty}\frac{A(n)}{n^{d/h}}=0.\]
\end{cnj}
On the $\limsup$ side, the $B_h$-sets of Wang \& Yang are $B_h[g]$-sets, so $\limsup_n A(n)/n^{1/h}\ge1/h$ is attainable for every $g$ when $d=1$; the question is how the attainable $\limsup$ grows with $g$. For $d=1$ and $h=2$, Part I~\cite{2026.Obryant-a} shows that the supremum over $g$-Golomb rulers lies between $\sqrt{g/2}$ and $\sqrt g$, and the finite constructions above suggest that for $B_h[g]$-sets the answer should be of order $g^{1/h}$.
\begin{qst}{Dense scales for $B_h[g]$-sets}{Bhg limsup}
  For positive integers $d,h,g$ with $h\ge2$, let $\sigma_{d,h,g}$ be the supremum of $\limsup_{n\to\infty}A(n)/n^{d/h}$ over all $B_h[g]$-sets $\cA\subseteq\NN^d$. Is $\sigma_{d,h,g}\gg_{d,h}g^{d/h}$?
\end{qst}

\section*{Tool and computational resource disclosure}
This work was developed in interaction with Anthropic's \emph{ClaudeAI}, the Fable model; in particular, the generalization from cubes in $\NN^d$ to convex bodies in $\ZZ^d$ was first sketched by \emph{ClaudeAI} in response to a prompt asking what would change, and the sketch was then checked, corrected, and written out by the human author.
Algebra, calculus, and inequalities were checked with Wolfram's \emph{Mathematica 14.3} and \emph{15.0.1}.
Lamport's \LaTeX\ was used both for typesetting and interacting with \emph{ClaudeAI}.
While the writing has been heavily influenced by \emph{ClaudeAI}, every line and implication has been understood, re-organized, and re-written by the human author, who takes responsibility for the correctness and clarity of this work.
Harmonic's \emph{AristotleAI} located a half-dozen typos and small errors (all now removed), and formally verified major portions of the one-dimensional version of this work.

\bibliographystyle{amsplain}
\bibliography{references}

\end{document}